\documentclass{amsart}
\usepackage[latin1]{inputenc}
\usepackage[active]{srcltx}
\usepackage{fullpage}
\usepackage{amsfonts,enumerate}
\usepackage[mathscr]{euscript}
\usepackage{epsfig}
\usepackage{latexsym}
\usepackage[all]{xy}
\usepackage{amssymb}
\usepackage{amsthm}
\usepackage{amsmath}
\usepackage{amsxtra}
\usepackage{placeins}
\usepackage{float}
\usepackage{xcolor}
\usepackage[colorlinks]{hyperref}
\hypersetup{citecolor=blue}

.tfm
.tfm

\theoremstyle{plain}
\newtheorem{prop}{Proposition}[section]
\newtheorem{thm}[prop]{Theorem}

\newtheorem{lemma}[prop]{Lemma}

\newtheorem*{thmA}{Theorem~\ref{thm:d=6}}
\newtheorem*{thmB}{Theorem~\ref{thm:cased=129}}

\theoremstyle{definition}

\theoremstyle{remark}

\newtheorem{remark}{Remark}

\numberwithin{table}{section}

\DeclareMathOperator{\chic}{\chi_{\text{cyc}}}

\DeclareMathOperator{\lcm}{lcm}

\DeclareMathOperator{\Id}{Id}

\DeclareMathOperator{\Cl}{Cl}
\DeclareMathOperator{\Gal}{Gal}

\newcommand{\D}{\mathcal D}

\newcommand{\bfP}{\mathcal P}

\newcommand{\F}{\mathbb F}

\newcommand{\E}{E_{(a,b,c)}}

\newcommand{\Om}{{\mathscr{O}}}

\newcommand{\Disc}{\Delta}

\newcommand{\GL}{{\rm GL}}

\newcommand{\PSL}{{\rm PSL}}

\def\ZZ{\mathbb Z}

\def\RR{\mathbb R}

\def\QQ{\mathbb Q}
\def\II{\mathbb I}

\def\<#1>{{\left\langle{#1}\right\rangle}}

\def\Z{{\mathbb Z}}             
\def\Q{{\mathbb Q}}             

\def\id#1{{\mathfrak{#1}}}      
\def\normid#1{{\norm{\id{#1}}}}
\DeclareMathOperator{\norm}{{\mathscr N}}

\let\kro\dkro

\begin{document}
	
\title{$\Q$-curves, Hecke characters and some Diophantine equations II.}
	
\author{Ariel Pacetti} \address{Center for Research and Development in
  Mathematics and Applications (CIDMA), Department of Mathematics,
  University of Aveiro, 3810-193 Aveiro, Portugal}
\email{apacetti@ua.pt}
	
\author{Lucas Villagra Torcomian} \address{FAMAF-CIEM, Universidad
  Nacional de C\'ordoba. C.P:5000, C\'ordoba, Argentina.}
\email{lucas.villagra@unc.edu.ar}


\begin{abstract}
  In the article \cite{PT} a general procedure to study solutions of
  the equations $x^4-dy^2=z^p$ was presented for negative values of
  $d$. The purpose of the present article is to extend our previous
  results to positive values of $d$. On doing so, we give a
  description of the extension
  $\Q(\sqrt{d},\sqrt{\epsilon})/\Q(\sqrt{d})$ (where $\epsilon$ is a
  fundamental unit) needed to prove the existence of a Hecke character
  over $\Q(\sqrt{d})$ with prescribed local conditions. We also extend some
  ``large image'' results due to Ellenberg regarding images of Galois representations
  coming from $\Q$-curves  from
  imaginary to real quadratic fields.
\end{abstract}
	
\maketitle
	
\section*{Introduction}
The study of solutions of Diophantine equations has been a very active
research field since Wiles' proof of Fermat's Last Theorem. There are
still many open conjectures on understanding solutions of a
generalized equation
\begin{equation}
  Ax^p +By^q = Cz^r,
\label{eq:generalizedfermat}
\end{equation}
for $\frac{1}{p} + \frac{1}{q} + \frac{1}{r} <1$. As was already
observed in \cite{MR1348707}, if no condition on the solutions is
imposed then the equation might have infinitely many solutions. To
overcome this subtlety we restrict to what in the literature is called
\emph{primitive solutions}. A solution $(a,b,c)$ to
(\ref{eq:generalizedfermat}) is called \emph{primitive} if the numbers
$\{aA,bB,cC\}$ are pairwise coprime.

A particular interesting example of~(\ref{eq:generalizedfermat})
occurs for exponents $(p,q,r) = (4,2,r)$ and $(A,B,C) = (1,1,1)$,
studied by Darmon and Ellenberg independently (see
\cite{MR2075481}). The Frey curve attached to a solution of it happens
to be a $\Q$-curve (i.e. an elliptic curve defined over a number
field, which is isogenous to all its Galois conjugates). $\Q$-curves
have the special property that a twist of their Galois representation
extends to a Galois representation of the whole Galois group
$\Gal(\overline{\Q}/\Q)$ and by \cite[Theorem 4.4]{MR2058653} and
Serre's modularity conjecture (\cite{MR885783}, \cite{MR2551764} and \cite{MR2827796})
it equals the Galois representation of a classical modular form. Then,
one can follow the modular method to compute (via a lowering the level
argument) a fixed space of level $N$ and weight two modular forms
(with a Nebentypus $\varepsilon$) and try to discard the ones that
cannot match a possible solution (due to a so called ``local''
obstruction). Using this method, in \cite{PT} the equation
\begin{equation}
  \label{eq:24p}
  x^4-dy^2=z^p
\end{equation}
was studied for different negative values of $d$. The novelty was to
use the theory of Hecke characters over imaginary quadratic fields to
give a precise formula for the value of $N$ and the character
$\varepsilon$. A natural question is the following: what happens if we
take positive values of $d$?

To a primitive solution $(a,b,c)$ of (\ref{eq:24p}) (or equivalently a
solution $(a,b,c)$ satisfying that the values $\{a,b,c\}$ are pairwise
coprime), one associates (as explained in \cite{MR2561200}) the
elliptic curve
\[
  \E: y^2=x^3+4ax^2+2(a^2+\sqrt{d}b)x, 
\]
defined over the field $K = \Q(\sqrt{d})$. When $d$ is positive
(and not a square) $K$ is a real quadratic field. It is known that all
elliptic curves over real quadratic fields are modular
(see \cite{MR3359051}) hence one can follow the modular approach working
with Hilbert modular forms. It turns out that such approach becomes
impractical very soon, due to the huge dimension of the corresponding
spaces (see Table \ref{table:examplesd}). However, the $\Q$-curves approach is still practical in many
circumstances, which motivates the present article. This article
should be thought as a continuation of our previous work \cite{PT}, where we
settle the following problems:
\begin{itemize}
\item Prove the existence of Hecke characters over real quadratic
  fields with prescribed local behavior.
		
\item Give a precise recipe for the level $N$ and the Nebentypus
  $\varepsilon$.
		
\item Show how Ellenberg's ``large image'' result can be adapted
  (under some hypothesis) to real quadratic fields and how it can be
  used to discard modular forms with complex multiplication.
		
\item Explain why the case $d$ positive is harder due to potential
  existence of non-trivial primitive solutions for all exponents $p$.
\end{itemize}
Section~\ref{sec:examples} contains different examples aiming to
explain the difference between the Hilbert/$\Q$-curves computational
effort. We also explain why in some cases there exist non-trivial
solutions of (\ref{eq:24p}) with $c= \pm 1$, which are valid for all
exponents $p$, making the modular approach fail. At last, we
explain why when there are modular forms with complex multiplication to be discarded,
classical results give a partial result for all primes satisfying some
congruence. We provide an example ($d = 3 \cdot 43$) where Ellenberg's
large image result applies, and a non-existence result for all large
enough primes can be obtained.

The article is organized as follows:
Section~\ref{section:modularmethod} contains a quick review of the
strategy developed in \cite{PT} as well as a review of the modular method.
In Section~\ref{sectionHeckeChar} (Theorem~\ref{thm:charexistence}) we solve the first problem described
above, namely the existence of a Hecke character with the desired
properties. The good definition of the character is related to a very
interesting problem of class field theory, namely suppose that
$K=\Q(\sqrt{d})$ is a real quadratic field, and $\epsilon$ is a
totally positive fundamental unit congruent to $1$ modulo $8$ (such
assumption is for expository purposes only, we consider the general
case in the article). Then the extension $K(\sqrt{\epsilon})$ is a
quadratic unramified extension of $K$, hence by class field theory it
corresponds to a genus character (see for example Chapter 2 of
\cite{MR3236783}). Is there a natural description for such character?
Can the extension $K(\sqrt{\epsilon})$ be described in terms of $d$?
	
We give a positive answer to this problem (Theorem~\ref{thm:epsilondescription}), which plays a crucial role
in the proof of the good definition of our Hecke character. The third
section (Theorem~\ref{thm:levelandnebentypus}) settles the second issue, namely it gives a precise recipe for
$N$ and $\varepsilon$. A proof of such statement was given in
\cite{PT} when $K$ is imaginary quadratic, since the Nebentypus had a
unique candidate due to the fact that it was odd. For real quadratic
fields, the hard part is to prove the formula for the Nebentypus! We
do so by computing explicitly an action on $3$-torsion points. The
proof might be of independent interest.
	
The fourth section gives an explicit version of Ellenberg's large image result for
real quadratic fields where the prime $2$ splits. The proof follows
from an ``explicit'' version of the main result of \cite{LeFourn}; our
little contribution being making the constants explicit. The last
section contains the examples, where the cases $d=6$ and $d=129$ are
specially considered along with other values of $d$ between $1$ and
$20$ (see Table \ref{table:examplesd}). Here are two instances of the results proved in the present article:
	
\begin{thmA} Let $p>19$ be a prime number such that $p\neq 97$ and
  $p\equiv1,3\pmod8$. Then, $(\pm 7,\pm20,1)$ are the only non-trivial primitive
  solutions of the equation
  \[ x^4-6y^2=z^p.
  \]
\end{thmA}

\begin{thmB} Let $p>19$ be a prime number satisfying that either
  $p>900$ or $p\equiv 1,3\pmod 8$ and $p \neq 43$. Then there are no non-trivial primitive
  solutions of the equation
  \[ x^4-129y^2=z^p.
  \]
\end{thmB}
	
We want to remark that the techniques and methods developed in the
present article can be used to study the equation $x^2-dy^6 = z^p$ for
positive values of $d$ following the results of \cite{PT}. The code
in \verb*|PARI/GP| (\cite{PARI2}) and \verb*|Magma| (\cite{MR1484478}) used in the examples (and the outputs), as well the one used to verify
Tables~\ref{table:cased3and7}, \ref{table:cased11and15} and
\ref{table:cased2}, are available at the web page
\url{https://github.com/lucasvillagra/Q-curves2.git}.
	
\subsection*{Acknowledgments} We would like to thank Yingkun Li for
sharing with us a proof of Theorem~\ref{thm:epsilondescription} and to
Harald Helfgott for providing some bounds used in
Section~\ref{section:Ellenberg}. We also want to thank Nuno Freitas
for some useful comments on the proofs of
Section~\ref{sec:nebentypus}. At last, we want to thank the referee
for his/her suggestions that improved the quality of the present
article. This research was partially supported by FonCyT BID-PICT
2018-02073 and by the Portuguese Foundation for Science and Technology
(FCT) within project UIDB/04106/2020 (CIDMA).

\section{Brief review of the modular method}\label{section:modularmethod}

Let us recall briefly how the modular method works. To a putative
primitive solution $(a,b,c)$ of~(\ref{eq:24p}), attach the elliptic
curve $\E$  given by the equation
\begin{equation}
  \label{eq:freycurve}
  \E: y^2=x^3+4ax^2+2(a^2+\sqrt{d}b)x,
\end{equation}
defined over the quadratic field $K = \Q(\sqrt{d})$. Let $\Gal_K$
denote an absolute Galois group of $K$, i.e.
$\Gal_K: = \Gal(\overline{\Q}/K)$ and for $p$ a prime number, let
$\rho_{\E,p}: \Gal_K \to \GL_2(\Z_p)$ denote the $2$-dimensional
$p$-adic Galois representation attached to $\E$ (obtained by looking
at the action of the Galois group on the $p$-adic Tate module of the
curve $\E$).  The curve $\E$ is what is called a $\Q$-curve, namely
its Galois conjugate is isogenous (via the order $2$ isogeny whose
kernel is the point $(0,0)$) to itself (see for example
\cite[Proposition 2.2]{PT}). The problem is that the isogeny is not
defined over $K$ but over $K(\sqrt{-2})$, so the Galois representation
$\rho_{\E,p}$ does not extend to a $2$-dimensional representation of
the whole Galois group $\Gal_\Q:=\Gal(\overline{\Q}/\Q)$.  However,
there exists a character $\chi$ (that will be constructed in the next
section) such that the twisted representation
$\rho_{\E,p} \otimes \chi$ does extend to an odd two dimensional
Galois representation of the whole Galois group
$\Gal(\overline{\Q}/\Q)$. Let $\tilde{\rho_p}$ denote such an
extension.

It is well known that modularity of the representation
$\tilde{\rho}_p$ follows from Serre's modularity conjecture (see
\cite[Theorem 4.4]{MR2058653}). As a side remark, Ribet's proof uses
the fact that our representation is related to an abelian variety of
$\GL_2$-type. Modularity of a two dimensional odd abstract
representations (satisfying the usual geometric hypothesis) is also
known if $p \ge 5$ (see \cite[Theorem 1.0.4]{1901.07166}). In
particular, $\tilde{\rho}_p$ matches the Galois representation of a
weight $2$, level $\tilde{N}$ and Nebentypus $\varepsilon$ newform
$f_{(a,b,c)}$ (the level and Nebentypus are described explicitly in
Theorem~\ref{thm:levelandnebentypus}).

The classical Hellegouarch result implies that our residual
representation $\overline{\rho_{\E,p} \otimes \chi}$ is unramified at
all primes not dividing $2d$, and the same holds for
$\overline{\tilde{\rho}}$. Suppose that $p$ is a prime number such
that the residual representation of $\tilde{\rho}_p$ is absolutely
irreducible. Then Ribet's lowering the level result (\cite{MR1104839})
implies that we have a congruence modulo $p$ between our newform
$f_{(a,b,c)}$ and a newform $g_{(a,b,c)}$ whose level $N$ is only
divisible by primes dividing $2d$ and with the same Nebentypus. We are
now led to ``discard'' the newforms $g \in S_2(N,\varepsilon)$ that do
not come from real solutions.

The first elimination process consists in applying the so called
``Mazur's trick'', namely check whether the eigenvalues are consistent
with a ``local'' solution of the original equation. More concretely,
suppose we intend to discard a form $g$. Let $q$ be a prime number
such that $q\nmid 2pd$, and let
\[
C(q,g) = \prod_{(a,b,c) \in \F_q^3}B(q,g;a,b,c),
\]
where the product is over non-zero triples $(a,b,c)$ satisfying~(\ref{eq:24p})
modulo $q$, and where the number $B(q,g;a,b,c)$ is defined by
  \[
    B(q,g;a,b,c)=\begin{cases}
      \norm(a_{\id{q}}(\E)\chi(\id{q})-a_q(g)) & \text{ if } q \nmid c \text{ and } q \text{ splits as }q = \id{q}\overline{\id{q}},\\
      \norm(a_q(g)^2-a_q(\E)\chi(q)-2q\varepsilon(q)) & \text{ if } q \nmid c \text{ and } q \text{ is inert in }K,\\
      \norm(\varepsilon^{-1}(q)(q+1)^2-a_q(g)^2) & \text{ if }q \mid c.
    \end{cases}
  \]
  If $(a,b,c)$ is a solution of (\ref{eq:24p}) and
  $g \in S_2(N,\varepsilon)$ is congruent modulo $p$ to
  $f_{(a,b,c)}$, it must be the case that $p\mid C(q,g)$ for all prime
  numbers $q$ (see \cite[Proposition 6.1]{PT}). We say that the form
  $g$ passes the test if $C(q,g) \neq 0$ for some small prime $q$.  If all
  newform pass the test, we can conclude that no such a solution exists
  (which never happens, due to the existence of a trivial solution).
	
  If $(a,b,c)$ is a solution of equation~(\ref{eq:24p}) for all primes
  $p$, and $g \in S_2(N,\varepsilon)$ is the modular form
  congruent modulo $p$ to $f_{(a,b,c)}$ then $C(q,g)=0$ for all primes
  $q$, so the above method fails. This occurs precisely when
  $c = \pm 1$. When $d<0$, the only solutions with $c = \pm 1$ are the
  trivial ones, but the Frey curves $E_{(\pm1,0,1)}$ have complex
  multiplication. To discard forms with complex multiplication
  Ellenberg's result ({\cite[Theorem 3.14]{MR2075481}}) is
  needed. Modular forms with complex multiplication have the property
  that the image of their Galois representations are not as large as
  expected (their image lies in the normalizer of a Cartan group),
  while for $\Q$-curves without complex multiplication Ellenberg's
  results implies that their projective residual image contain
  $\PSL_2(\F_p)$, hence they cannot be congruent.  This is the reason
  why we could prove non-existence of non-trivial primitive solutions
  of (\ref{eq:24p}) for different negative values of $d$ in \cite{PT}.
	
  There are two unfortunate situations where the previous approach
  cannot be applied. One of them is when \cite[Theorem
  3.14]{MR2075481} cannot be applied. Then we can only hope to prove
  non-existence of solutions for primes satisfying certain congruence
  properties (the ones where the curve coming from the trivial
  solution has small image, namely its residual image is contained in
  the normalizer of a split Cartan subgroup). The second one (which
  only occurs when $d>0$) is when the curve
	\begin{equation} x^4-dy^2=\pm1
		\label{eq:C=1}
	\end{equation} admits non-trivial solutions. For $1< d < 20$,
the non-trivial solutions of such an equation are precisely the following
\begin{equation}\label{eq:solutionsc=1}
(a,b,c,d)\in\{(\pm1,\pm1,-1,2),
(\pm3,\pm4,1,5),(\pm7,\pm20,1,6),(\pm2,\pm1,1,15),(\pm2,\pm1,-1,17)\}.
\end{equation}
	 
	 Equation ~(\ref{eq:C=1}) was studied in several articles
(see for example \cite{MR1770484}). It is known that the equation with
$+1$ on the right hand side has at most one non-trivial solution
(see \cite{MR12619}) except when $d=1785$. Furthermore, in
\cite{MR1438594} all solutions for $1\leq d \leq 150000$ are
computed. The equation with $-1$ on the right hand side was studied in
\cite{MR0065575}, where it is also shown that in all cases there is at
most one non-trivial solution, and a condition for the existence is
presented. A priori, the modular method should not work in cases when
there exists a solution  of (\ref{eq:C=1}) (although we
will soon prove it does work for $d=6$).

\section{Construction of the Hecke character}\label{sectionHeckeChar}

Given $\tau \in \Gal_\Q$ and $\rho$ a representation of $\Gal_K$, by
$^\tau\rho$ we denote the representation of $\Gal_K$ whose value at
$\sigma \in \Gal_K$ equals
\[
^\tau\rho(\sigma) = \rho(\tau \sigma \tau^{-1}).
\]
Fix an element $\tau \in \Gal_\Q$ which is not the identity on
$K$. Then the curve $\tau(\E)$ is isogenous to
$\E \otimes \delta_{-2}$ (the quadratic twist of the curve by $-2$) as
proved in \cite[Proposition 2.2]{PT}. This implies that 
\begin{equation}
  \label{eq:transformation}
^\tau \rho_{\E,p} \simeq \rho_{\E,p} \otimes \delta_{-2},  
\end{equation}
where we interpret $\delta_{-2}$ as the quadratic character of
$\Gal_K$ corresponding (via class field theory) to the quadratic
extension $K(\sqrt{-2})/K$. Note that $\delta_{-2}$ is actually a
quadratic character of $\Gal_\Q$ restricted to $\Gal_K$.

\begin{remark}
  All the previous stated properties hold for any pair of rational
  numbers $(a,b)$ (independently on whether they are part of a
  solution of (\ref{eq:24p}) or not). The fact that they are a
  solution is needed while studying the Kodaira type at bad primes,
  and also (together with the extra hypothesis that the solution is
  primitive) to assure that the residual representation
  $\overline{\rho_{\E,p}}$ is unramified at all prime ideals not
  dividing $2$.
\end{remark}

The main idea of \cite{PT} is to construct a finite order Hecke
character $\chi$ satisfying also property~(\ref{eq:transformation})
(using class field theory, we will denote indistinctly Hecke
characters and their Galois characters counterparts). If
$\chi: \Gal_K \to \overline{\Q}^\times$ is a Hecke character satisfying
$^\tau\chi(\sigma):=\chi(\tau \sigma \tau^{-1}) = \chi(\sigma)
\delta_{-2}(\sigma)$ for all $\sigma\in \Gal_K$, then the twisted representation
$\rho_{\E,p} \otimes \chi$ is invariant under the action of $\tau$ and
hence extends to a $2$-dimensional representation of $\Gal_\Q$.  How
can we construct a Hecke character $\chi$ on the id\`ele group of $K$
(that we denote $\II_K$) satisfying that
$^\tau\chi = \chi \cdot \delta_{-2}$?

Let $\Om_K$ denote the ring of integers of $K$, and given $\id{q}$ a
prime ideal of $\Om_K$, let $\Om_{\id{q}}$ denote the completion of
$\Om_K$ at $\id{q}$. Let $\Cl(K)$ denote the class group of $K$. From
the short exact sequence
\begin{equation}
  \label{eq:ideleSES} \xymatrix{ 0 \ar[r] & K^\times
    \cdot(\prod_{\id{q}}\Om_{\id{q}}^\times \times (\RR^\times)^2) \ar[r]
    \ar[r] & \II_K \ar[r]^{\Id} & \Cl(K) \ar[r] & 0, }
\end{equation}
it is enough to define the character $\chi$ on
$\prod_{\id{q}}\Om_{\id{q}}^\times \times (\RR^\times)^2$, on
$K^\times$ (where the character is trivial) and on id\`eles
representing the class group of $K$ (i.e. elements of $\II_K$ that are
in bijection with representatives for the class group $\Cl(K)$ under
the map $\Id$). The intersection of these two subgroups
$(\prod_{\id{q}}\Om_{\id{q}}^\times \times (\RR^\times)^2) \cap
K^\times = \Om_K^\times$ imposes a \emph{compatibility} condition on
its definition, namely the product of the local components evaluated
at a unit equals $1$. When $d>0$ the ring
$\Om_K^\times = \langle -1,\epsilon\rangle$, where $\epsilon$ denotes
a fundamental unit, hence it is enough to check compatibility at both
such elements. The compatibility was proven in \cite[Theorem 3.4]{PT}
when the fundamental unit has norm $-1$, so, after replacing
$\epsilon$ by $-\epsilon$ if needed, we assume that $\epsilon$ is
totally positive.
	
Let us briefly recall the construction given in \cite{PT} (there is a
discrepancy with the definitions used in \cite{PT}, namely $d$ needs
to be changed to $-d$ in such article). Split the odd prime divisors
of $d$ into four different sets, namely:
\[ Q_i =\{ p \text{ prime}\; : \; p \mid d, \quad p \equiv i \pmod
  8\},
\] for $i=1, 3, 5, 7$. Let $\delta_{-1}$, $\delta_{2}$, $\delta_{-2}$
be the characters of $\ZZ$ corresponding to the quadratic extensions
$\QQ(\sqrt{-1})$, $\QQ(\sqrt{2})$ and $\QQ(\sqrt{-2})$ respectively
and (abusing notation) let $\delta_{-1}$, $\delta_{2}$, $\delta_{-2}$
also denote their local component at the prime $2$. Define a character
$\varepsilon:\II_\Q \to \overline{\Q}^\times$ (that will be the
Nebentypus of the extended Galois
representation) as follows:
\begin{itemize}
\item For primes $p\nmid d$ and  also for primes $p \in Q_1 \cup Q_7$, the character
  $\varepsilon_p:\ZZ_p^\times \to \overline{\Q}^\times$ is trivial.
		
\item For primes $p \in Q_3$, the character
  $\varepsilon_p(n) = \kro{n}{p}$ (quadratic).
		
\item For $p \in Q_5$, let $\varepsilon_p$ be a character of order $4$
  and conductor $p$.
\item The character $\varepsilon_\infty$ (the archimedean component)
  is trivial.
\item Define $\varepsilon_2=\delta_{-1}^{\#Q_3+\#Q_5}$.
\end{itemize}
Since $\Q$ has class number one, the rational id\`eles $\II_\Q$ is
isomorphic to
$\Q^\times \cdot (\prod_p \Z_p^\times \times \RR^\times)$, hence our
local definitions give rise to a unique Hecke character $\varepsilon$
once the compatibility condition is checked. But
\[
\prod_p\varepsilon_p(-1)\varepsilon_{\infty}(-1) = \prod_{p \in Q_3\cup Q_5}\varepsilon_p(-1) \varepsilon_2(-1) = (-1)^{\#Q_3+\#Q_5}\varepsilon_2(-1)=1.
\]
By class field theory, $\varepsilon$ gets identified with a character
$\varepsilon:\Gal_\Q \to \overline{\Q}^\times$ whose kernel fixes a
totally real field $L$ whose degree equals $1$ if $Q_3=Q_5=\emptyset$,
$2$ if $Q_3\neq Q_5 = \emptyset$ and $4$ otherwise.  Let
$N_\varepsilon$ denote its conductor, given by
$N_\varepsilon= 2^e\prod_{p \in Q_3 \cup Q_5}p$, where $e=0$ if
$\#Q_3+\#Q_5$ is even and $2$ otherwise. If $p$ is an odd prime
dividing $d$, we denote by $\id{p}$ the unique prime in $K$ dividing
it.
\begin{thm} There exists a
  Hecke character $\chi:\Gal_K \to \overline{\QQ}^\times$ such that:
  \begin{enumerate}
  \item $\chi^2 = \varepsilon$ as characters of $\Gal_K$,
			
  \item $\chi$ is unramified at primes not dividing
    $2 \cdot \prod_{p \in Q_1 \cup Q_5 \cup Q_7}p$,
			
  \item for $\tau$ in the above hypothesis,
    $^\tau\chi = \chi \cdot \delta_{-2}$ as characters of $\Gal_K$.
  \end{enumerate}
  \label{thm:charexistence} Furthermore, if $d$ denotes the
  discriminant of $K$ and $\id{p}_2$ is a prime of $K$ dividing $2$, then its conductor equals
  $\id{p}_2^e \cdot \prod_{p \in Q_1 \cup Q_5 \cup Q_7} \id{p}$, where
  \[ e=
    \begin{cases}
      5 & \text{ if } d/4 \equiv 7 \pmod 8,\\
      3 & \text{ if \; \;} d \equiv 1 \pmod 4,\\
      3 & \text{ if } d/4 \equiv 2, 3 \pmod {8},\\
      3 & \text{ if } d/4 \equiv 6 \pmod {16},\\
      0 & \text{ if } d/4 \equiv 14 \pmod {16}.\\
    \end{cases}
  \]
\end{thm}
The theorem was proved in \cite{PT} (Theorem 3.2) for $d<0$
and for $d>0$ when the fundamental unit $\epsilon$ has norm $-1$. The
main obstacle in the remaining case is to have some understanding on
the reduction of a positive fundamental unit modulo ramified primes of
$K$. Let us state the following related natural problem.
	
\medskip
	
\noindent {\bf Problem:} Let $K/\Q$ be a real quadratic field, and let
$\epsilon$ be a totally positive fundamental unit. What can be said of
the extension $K(\sqrt{\epsilon})/K$?  \medskip
	
Suppose that $K=\Q(\sqrt{d})$ with $d$ a positive fundamental 
discriminant (i.e. equals the discriminant of the extension $K/\Q$).
Let $p \mid d$ be an odd prime and let $\id{p}$ denote
the unique prime ideal of $K$ dividing it. The hypothesis
$\norm(\epsilon) = 1$ implies that
$\epsilon \equiv \pm 1 \pmod{\id{p}}$. Let
	\[ \bfP_{\pm} = \{p \mid d, \; p\text{ odd } \; : \;
\epsilon \equiv \pm 1 \pmod{\id{p}}\}.
\]
If $2$ ramifies in $K/\Q$, let
$\id{p}_2$ denote the unique prime of $K$ dividing it.
\begin{thm} Let $\omega:= \prod_{p \in \bfP_{-}} p$. Then:
  \begin{itemize}
  \item if $2$ is unramified in $K/\Q$, we have
    $K(\sqrt{\epsilon}) = K(\sqrt{\omega})$,
  \item if $2$ is ramified in $K/\Q$, we have
    $K(\sqrt{\epsilon}) = K(\sqrt{2\omega})$ or
    $K(\sqrt{\epsilon}) = K(\sqrt{\omega})$.
  \end{itemize}
Furthermore, when $8 \mid d$,
the latter case occurs precisely when
$\epsilon \equiv -1 \pmod{\id{p}_2^3}$.
  \label{thm:epsilondescription}
\end{thm}
\begin{proof} Let us recall some well known results on the narrow
  class group of a real quadratic field. The result is due mostly to
  Gauss \cite{MR837656} (see also \cite{MR1012948} for a more modern
  presentation), although Gauss' approach was via the study of
  indefinite binary quadratic forms. Among such forms, there are
  some special ones called ``ambiguous forms'' (see \cite{MR1012948}
  page 7 Chapter 1 and page 24 Chapter 3), which are precisely the
  elements of order two under Gauss' composition law. The total number
  of ambiguous classes (including the trivial one) equals $2^{t-1}$,
  where $t$ is the number of prime divisors of $d$ (by
  \cite[Proposition 4.7]{MR1012948} and its proof).

  Recall that there is a correspondence between strict equivalence
  classes of indefinite binary quadratic forms of discriminant $d$ and
  ideal classes for the narrow class group of $K$. Under this
  correspondence, the ambiguous forms map to ideals of order two in
  the narrow class group. But such ideals correspond precisely to the
  ramified prime ideals of $K$ (indexed by divisors of $d$), by
  \cite[Corollary 4.9]{MR1012948}. In particular, there exists a
  unique non-trivial and square-free principal ideal $\id{d}$
  (generated by a totally positive element $\alpha$) dividing the
  different $\D$ of $K$. Let
  $\omega:=\normid{d} = \norm(\alpha)= \alpha \overline{\alpha}$, so
  that $\omega \mid d$.
		
  Since all ramified primes are invariant under conjugation, and
  $\id{d}$ is divisible only by ramified primes,
  $\overline{\id{d}} = \id{d}$. Then the quotient
  $\frac{\alpha}{\overline{\alpha}} \in \Om_K$ is a totally positive
  unit which cannot be trivial (as otherwise $\alpha \in \Q_{>0}$, but
  it must divide the different of $K$ and also generate a square-free
  ideal of $\Om_K$, hence equals $1$). Substituting $\alpha$ by
  $\epsilon^k \alpha$ changes the quotient
  $\frac{\alpha}{\overline{\alpha}}$ by a factor of
  $\epsilon^{2k}$, so we can assume that
  \begin{equation}
    \frac{\alpha}{\overline{\alpha}} = \epsilon.
  \end{equation}
  Then
  $\sqrt{\epsilon} = \frac{\sqrt{\alpha
      \overline{\alpha}}}{\overline{\alpha}}$ and hence
  $K(\sqrt{\epsilon}) = K(\sqrt{\omega})$. We are led to determine the
  set of primes dividing $\omega$. Let $\id{p}$ be a prime ideal dividing
  $\D$ and assume that $\id{p} \nmid 2$.
  \begin{itemize}
			
  \item The fact that
    $\alpha + \overline{\alpha} \in \id{d} \cap \ZZ = (\omega)$ (which generates over $K$ the ideal $\id{d}^2$) implies
    that $\alpha + \overline{\alpha} \in \id{d}^2$, hence
    $\epsilon +1 = \frac{\alpha}{\overline{\alpha}}+1 = \frac{\alpha +
      \overline{\alpha}}{\overline{\alpha}} \in \id{d}$ and then
    $\epsilon \equiv -1 \pmod{\id{d}}$. In particular,
    $\epsilon \equiv -1 \pmod{\id{p}}$ for all odd prime ideals
    $\id{p} \mid \id{d}$.
			
  \item On the other hand, if $\id{p} \mid \D$ but $\id{p} \nmid \id{d}$
    (in particular $\id{p} \nmid \overline{\alpha}$),
    $\epsilon-1 = \frac{\alpha - \overline{\alpha}}{\overline{\alpha}} \equiv 0
    \pmod{\id{p}}$ hence $\epsilon \equiv 1 \pmod{\id{p}}$.
  \end{itemize}
  If $2 \nmid d$ then $\omega = \prod_{p \in \bfP_{-}} p$ and the statement follows.
  If $d$ is even the only ambiguity is whether $\omega$ is even or
  not. Suppose that $8 \mid d$. Let $\id{p}_2$ denote the prime ideal
  dividing $2$ ($\id{p}_2 = \langle2, \sqrt{d/4}\rangle$). Clearly
  $v_{\id{p}_2}(\alpha) = v_{\id{p}_2}(\bar{\alpha})=v_2(\omega)$. An
  elementary case by case analysis shows that
  $v_{\id{p}_2}(\alpha) \in \{0,2\}$ if and only if
  $v_{\id{p}_2}(\epsilon -1) \ge 3$ and $v_{\id{p}_2}(\epsilon
  +1)=2$. Similarly, $v_{\id{p}_2}(\alpha) \in \{1,3\}$ if and only if
  $v_{\id{p}_2}(\epsilon +1) \ge 3$ and $v_{\id{p}_2}(\epsilon -1) = 2$ as
  stated.
\end{proof}
	
\begin{proof}[Proof of Theorem~\ref{thm:charexistence}] Keeping the
  previous notation, let $d$ denote the discriminant of $K$. Let
  $\chi_\id{p}:\Om_{\id{p}}^\times \to \overline{\Q}^\times$ be the character given
  by
  \begin{itemize}
  \item If $\id{p}$ is an odd (i.e. $\id{p}\nmid 2$) unramified prime,
    $\chi_{\id{p}}$ is the trivial character. The same applies to
    primes in $K$ dividing the primes in $Q_3$.
			
  \item If $p$ is an odd prime ramifying in $K/\Q$ and
    $\id{p} \mid p$, clearly
    $(\Om_{\id{p}}/\id{p})^\times \simeq (\ZZ/p)^\times$. If
    $p \in Q_1 \cup Q_7$, let $\chi_{\id{p}}$ correspond to the
    quadratic character $\delta_p$ of $(\ZZ/p)^\times$.
  \item If $p \in Q_5$, using the previous item isomorphism, let
    $\chi_{\id{p}}= \varepsilon_p \cdot \delta_p$.
  \end{itemize} At the archimedean places $\{v_1,v_2\}$, let
  $\chi_{v_1}$ be the trivial character and $\chi_{v_2}$ be the sign
  function (the order of the archimedean places does not matter, both
  choices work). At a prime $\id{p}_2$ dividing $2$, the character
  $\chi_{\id{p}_2}$ has conductor at most $2^3$. The group structure of
  $(\Om_{\id{p}_2}/2^n)^\times$ and its generators when $2$ does not
  split are given in Table~\ref{table:quotientstructure} (see
  \cite{Ranum1910}). The generators are ordered so that the order of
  the generator $i$ matches the $i$-th factor of the group structure,
  while the elements norms are modulo $8$.

  \begin{table}[H]
  \begin{tabular}{|c|c|c|c|c|}
    \hline
    Condition & $n$ & Structure & Generators & Norms \\
    \hline\hline
    $d \equiv 5 \pmod 8$ & $3$ & $\Z/3 \times \Z/4 \times \Z/2 \times \Z/2$ &$\{\zeta_3,\sqrt{d},3+2\sqrt{d},-1\}$ & $\{1,3,5,1\}$\\
    \hline
    $d/4 \equiv 7 \pmod 8$ & $3$ & $\Z/4 \times \Z/4 \times \Z/2$& $\{\sqrt{d/4}, 1+2\sqrt{d/4},5\}$ & $\{1,5,1\}$ \\
    \hline
    $d/4 \equiv 3 \pmod 8$ & $3$ & $\Z/4 \times \Z/4 \times \Z/2$ &$\{\sqrt{d/4},1+2\sqrt{d/4},-1\}$ & $\{5,5,1\}$\\
    \hline
    $8 \mid d$ & $2$& $\Z/4 \times \Z/2 $ & $\{1+\sqrt{d/4},-1\}$ & $\{3,1\}$ \\
    \hline  
  \end{tabular}
  \caption{\label{table:quotientstructure}}
\end{table}
The definition of $\chi_{\id{p}_2}$ on this set of generators for
$(\Om_{\id{p}_2}/2^3)^\times$ is the following:
  \begin{itemize}
  \item If $d \equiv 5 \pmod 8$, $\chi_{\id{p}_2}(\zeta_3)=1$,
    $\chi_{\id{p}_2}(\sqrt{d})=i$, $\chi_{\id{p}_2}(3+2\sqrt{d})=1$,
    $\chi_{\id{p}_2}(-1)=1$.
  \item If $d/4 \equiv 7 \pmod {16}$, $\chi_{\id{p}_2}(\sqrt{d/4})=-1$,
    $\chi_{\id{p}_2}(1+2\sqrt{d/4})=1$, $\chi_{\id{p}_2}(5)=-1$.
  \item If $d/4 \equiv 15 \pmod {16}$, $\chi_{\id{p}_2}(\sqrt{d/4})=1$,
    $\chi_{\id{p}_2}(1+2\sqrt{d/4})=1$, $\chi_{\id{p}_2}(5)=-1$.
			
  \item If $d/4 \equiv 3 \pmod {16}$, $\chi_{\id{p}_2}(\sqrt{d/4})=-1$,
    $\chi_{\id{p}_2}(1+2\sqrt{d/4})=1$, $\chi_{\id{p}_2}(-1)=-1$.
			
  \item If $d/4 \equiv 11 \pmod {16}$, $\chi_{\id{p}_2}(\sqrt{d/4})=1$,
    $\chi_{\id{p}_2}(1+2\sqrt{d/4})=1$, $\chi_{\id{p}_2}(-1)=-1$.
			
  \item If $d/4 \equiv 6 \pmod 8$ and $\#Q_3+\#Q_5$ is even,
    $\chi_{\id{p}_2}(1+\sqrt{d/4})=1$, $\chi_{\id{p}_2}(-1)=1$,
    $\chi_{\id{p}_2}(5)=1$.
			
  \item If $d/4 \equiv 6 \pmod 8$ and $\#Q_3+\#Q_5$ is odd,
    $\chi_{\id{p}_2}(1+\sqrt{d/4})=i$, $\chi_{\id{p}_2}(-1)=-1$,
    $\chi_{\id{p}_2}(5)=1$.
			
  \item If $d/4 \equiv 2 \pmod 8$ and $\#Q_3+\#Q_5$ is even,
    $\chi_{\id{p}_2}(1+\sqrt{d/4})=1$, $\chi_{\id{p}_2}(-1)=-1$,
    $\chi_{\id{p}}(5)=1$.
			
  \item If $d/4 \equiv 2 \pmod 8$ and $\#Q_3+\#Q_5$ is odd,
    $\chi_{\id{p}_2}(1+\sqrt{d/4})=i$, $\chi_{\id{p}_2}(-1)=1$,
    $\chi_{\id{p}_2}(5)=1$.
  \end{itemize}
At last, 
\begin{itemize}
	\item If $d \equiv 1 \pmod 8$, the prime $2$ splits as
	$(2) = \id{p}_2 \overline{\id{p}}_2$. Let
	$\chi_{\id{p}_2}:=\delta_{-2}$ and $\chi_{\overline{\id{p}}_2}:=1$
	(trivial). 
\end{itemize}
Following the notation of \cite{PT}, we denote
$\chi_2=\prod_{\id{p}_2\mid2}\chi_{\id{p}_2}$.

  There are some constraints on the values of $\#Q_3$, $\#Q_5$ and
  $\#Q_7$ depending on the congruence of $d$ (or $d/4$) modulo $8$;
  they  are given in Table~\ref{table:parity}.

\begin{table}[h!]
	\begin{tabular}{|l|c|c|c||l|c|c|c|}
		\hline
		Condition & $\#Q_3$ & $\#Q_5$ & $\#Q_7$ &Condition & $\#Q_3$ & $\#Q_5$ & $\#Q_7$ \\
		\hline\hline
		$d \equiv 1 \pmod 8$ & $0$ & $0$ & $1$ &$d \equiv 5 \pmod 8$ & $0$ & $1$ & $1$\\
		& $1$ & $1$ & $0$ &  & $1$ & $0$ & $0$\\
		\hline
		$d/4 \equiv 3 \pmod 8$ & $0$ & $1$ & $0$ &$d/4 \equiv 7 \pmod 8$ & $0$& $0$ & $0$\\
		& $1$ & $0$ & $1$ &    & $1$ & $1$ & $1$\\
		\hline
		$d/4 \equiv 2 \pmod 8$ & $0$ & $0$ & $1$ & $d/4 \equiv 6 \pmod 8$ & $0$& $0$ & $0$\\
		& $0$ & $1$ & $1$ & & $0$ & $1$ & $0$\\
		& $1$ & $0$ & $0$ & & $1$ & $0$ & $1$\\
		& $1$ & $1$ & $0$ & & $1$ & $1$ & $1$ \\
		\hline
	\end{tabular}
	\caption{\label{table:parity}}
\end{table}

        Using such relations and the previous definitions, it is not
        hard to verify that in all cases
\begin{equation}
  \label{chi2}
\chi_2|_{\Z_2^\times}=\delta_{2}^{v_2(d)+1}\delta_{-1}^{\#Q_5+\#Q_7+1}.
\end{equation}
		
Extend $\chi$ to
$K^\times \cdot (\prod_{\id{q}}\Om_{\id{q}}^\times \times
(\RR^\times)^2)$ by making it trivial on $K^\times$. With these
definitions, the same proof given in \cite[Theorem 3.2]{PT} (page 14)
proves that the equality $\chi^2 = \varepsilon \circ \norm$ holds.
		
\noindent{\bf Compatibility:} the subgroup of units in $K$ is
generated by $\{-1,\epsilon\}$ hence it is enough to prove the
compatibility at both elements. Replacing $d$ by $-d$ we interchange
real quadratic fields with imaginary quadratic ones. The local part of
the character $\chi$ is invariant under such transformation for all
odd primes, but not at primes dividing $2$. For such places, the
restriction of the local character to $\Z_2^\times$ differs by
$\delta_{-1}$. In \cite[Theorem 3.2]{PT} we proved the compatibility
at $-1$ for imaginary quadratic fields $K$; since
$\delta_{-1}(-1) = -1$, the compatibility relation for real quadratic
fields at $-1$ follows from the extra sign coming from the archimedean
contribution.
		
Proving the compatibility for $\epsilon$ takes more effort.  The
character $\chi$ satisfies $\chi_{\id{p}}(\epsilon) = 1$ for all
unramified primes and for primes in $\bfP_-\cap( Q_1\cup Q_3)$ (recall
that the character $\chi_{\id{p}}$ has order $2$ at primes in $Q_1$
and is trivial at primes in $Q_3$). Its value at primes in
$\bfP_- \cap (Q_5 \cup Q_7)$ equals $-1$. Since the character
$\delta_{-2}$ also satisfies that it takes the value $-1$ at primes in
$Q_5 \cup Q_7$ and $+1$ at the other ones, we need to prove the
following identity
\begin{equation}
  \label{eq:compatibility} \chi_2(\epsilon)\cdot
  (-1)^{\#(\bfP_- \cap (Q_5 \cup
    Q_7))}=\chi_2(\epsilon)\delta_{-2}(\omega)=1,
\end{equation}
where $\omega = \prod_{p \in \bfP_{-}} p$
as before. The proof of Theorem~\ref{thm:epsilondescription} implies
that there exists $\alpha \in \Om_K$ such that $\omega=\epsilon
\overline{\alpha}^2$ or $2\omega=\epsilon \overline{\alpha}^2$. In the
first case,
\[
  \chi_2(\overline{\alpha}^2)=\chi_2^2(\overline{\alpha})=\varepsilon_2(\norm(\alpha))=\varepsilon_2(\omega).
\]
Since $\varepsilon_2$ is at most quadratic, it equals its inverse. Hence $\chi_2(\epsilon)=\chi_2(\omega)\varepsilon_2(\omega)$ and then equation
(\ref{eq:compatibility}) is equivalent to the statement
\begin{equation}
  \label{eq:comp2}
\chi_2(\omega)\varepsilon_2(\omega)\delta_{-2}(\omega) =1.
\end{equation} A key fact is that the hypothesis
$\norm(\alpha) = \omega$ imposes a constraint on its possible values.
Using equation~(\ref{chi2}), the proof follows from the following case
by case study:
		
\begin{itemize}
\item If $d \equiv 1 \pmod 8$, then $\chi_2 = \delta_{-2}$ and
  $\varepsilon_2$ is trivial hence (\ref{eq:comp2}) holds.
			
\item If $d/4 \equiv 3 \pmod 8$, the norm condition implies that $\omega$
  is congruent to $1$ or $5$ modulo $8$. By definition
  $\chi_2|_{\ZZ_2^\times}=\delta_{-2}$ and
  $\varepsilon_2=\delta_{-1}$, which is trivial on both $1, 5$ hence
  (\ref{eq:comp2}) holds.
			
\item If $d \equiv 5 \pmod 8$, by definition
  $\chi_2|_{\ZZ_2^\times}=\delta_{2}$ and $\varepsilon_2=\delta_{-1}$
  hence (\ref{eq:comp2}) holds.
			
\item If $d/4 \equiv 7 \pmod 8$, the norm condition implies that $\omega$
  is congruent to $1$ or $5$ modulo $8$. By definition
  $\chi_2|_{\ZZ_2^\times}=\delta_{2}$ and $\varepsilon_2=1$. But
  $\delta_2$ and $\delta_{-2}$ take the same values at $\{1,5\}$ hence
  (\ref{eq:comp2}) holds.
			
\item If $d/4 \equiv2\pmod8$, the norm condition implies that $\omega$ is
  congruent to $1$ or $7$ modulo $8$. By definition
  $\chi_2|_{\ZZ_2^\times}\cdot \varepsilon_2=\delta_{-1}$, which
  coincides with $\delta_{-2}$ on $\{1, 7\}$ hence (\ref{eq:comp2})
  holds.
			
\item If $d/4 \equiv6\pmod8$, the norm condition implies that $\omega$ is
  congruent to $1$ or $3$ modulo $8$. By definition
  $\chi_2|_{\ZZ_2^\times}\cdot \varepsilon_2=1$ but $\delta_{-2}$ is
  trivial on $\{1, 3\}$ hence (\ref{eq:comp2}) holds.
\end{itemize}
		
If $d$ is odd, the equality $\omega = \epsilon \bar{\alpha}^2$ always
holds hence the result follows. Assume then that $2$ ramifies in
$K/\Q$ and that $2\omega = \epsilon \bar{\alpha}^2$. Let $\id{p}_2$
denote the unique prime of $K$ dividing $2$. To ease notation, let
$\tilde{d} = d/4$. Recall that $K(\sqrt{\epsilon})$ is unramified at
$\id{p}_2$ if and only if $\epsilon$ is a square mod $4$ (see for
example \cite[Lemma 3.4]{MR3946721}). The equality
$2\omega = \epsilon \overline{\alpha}^2$ implies that
\begin{equation}
  \label{eq:tilded0}
\left(\frac{2}{\overline{\alpha}}\right)^2\omega = 2\epsilon.
\end{equation}
Note that $\frac{2}{\overline{\alpha}}$ has positive valuation at
$\id{p}_2$, hence we can reduce equality (\ref{eq:tilded0}) modulo
$16$ to compute for each possible value of $\epsilon$ the
corresponding value of $\omega$ (up to squares) via a finite
computation. Before presenting the results of the finite computation,
note the following: if $d_1 \equiv d_2 \pmod {16}$, then
$\Z[\sqrt{d_1}]/2^4 \simeq \Z[\sqrt{d_2}]/2^4$ (as rings) via the
natural map sending $\sqrt{d_1}$ to $\sqrt{d_2}$. Applying it to
equality (\ref{eq:tilded0}) proves that the value ${\omega}$ attached
to a fundamental unit of the form $a+b\sqrt{d_1}$ equals that attached
to $a+b\sqrt{d_2}$. In particular, it is enough to perform the finite
computation for $\tilde{d}$ modulo $16$.
		
If $\tilde{d} \equiv 3 \pmod 4$ and $t \mid d$ then the extension
$K(\sqrt{t})$ is ramified at $\id{p}_2$ precisely when $t$ is even
(and not divisible by $4$). Then under our hypothesis, the extension
$K(\sqrt{\epsilon})/K$ is ramified at $\id{p}_2$. Take
$\{\sqrt{d}/2,1+\sqrt{{d}},-1\}$ as generators for the group of
invertible elements modulo $16$ when $\tilde{d}\equiv3\pmod8$ and
$\{\sqrt{d}/2,1+\sqrt{{d}},5\}$ when
$\tilde{d}\equiv7\pmod8$. Consider the different cases, taking into
account once again that the condition $2\omega$ being a norm implies that
$\omega\equiv3,7\pmod8$ when $\tilde{d}\equiv3\pmod{8}$ and
$\omega\equiv1,5\pmod8$ when $\tilde{d}\equiv7\pmod{8}$. Then:
\begin{itemize}
\item If $\tilde{d} \equiv 3,7 \pmod {16}$, the possible values for
  $\epsilon$ (given as generators' exponents) and the values of
  ${\omega}$ are given in Table~\ref{table:cased3and7}.  Since
  $\chi_2((a,b,c)) = (-1)^{a+c}$ (again as exponents) the equality
  $\chi_2(\epsilon) = \delta_{-2}({\omega})$ follows recalling that
  $\delta_{-2}(1)=\delta_{-2}(3) = 1$ and
  $\delta_{-2}(5)=\delta_{-2}(7) = -1$.
  
  \begin{table}[H]
    \begin{tabular}{|c||c|c||c|c||c|c||c|c|}
    	\hline & & & & & & & &\\[-1em]
 $\tilde{d} \pmod{16}$ & Exp. & ${\omega}$ & Exp. & ${\omega}$ &
Exp. & ${\omega}$ & Exp. & ${\omega}$ \\ \hline $3$ & $(1,1,0)$ & $7$ &
$(1,1,1)$ & $3$ & $(1,3,0)$ & $ 7$ & $(1,3,1)$ & $3$\\ \hline $3$ &
$(3,1,0)$ & $7$ & $(3,1,1)$ & $3$ & $(3,3,0)$ & $7$ & $(3,3,1)$ &
$3$\\ \hline $7$ & $(1,0,0)$ & $5$ & $(1,0,1)$ & $1$ & $(1,2,0)$ & $5$
& $(1,2,1)$ & $1$\\ \hline $7$ & $(3,0,0)$ & $5$ & $(3,0,1)$ & $1$ &
$(3,2,0)$ & $5$ & $(3,2,1)$ & $1$\\ \hline
    \end{tabular}
    \caption{Relation between $\epsilon$ and ${\omega}$ for
      $\tilde{d} \equiv 3,7 \pmod{16}$ \label{table:cased3and7}}
  \end{table}
			
\item If $\tilde{d} \equiv 11,15\pmod{16}$ the possible values for
  $\epsilon$ and the values of ${\omega}$ are given in
  Table~\ref{table:cased11and15}. Since $\chi_2((a,b,c)) = (-1)^{c}$
  in this case, the equality $\chi_2(\epsilon)=\delta_{-2}({\omega})$
  holds.
  \begin{table}[h!]
  	\begin{tabular}{|c||c|c||c|c||c|c||c|c|}
  		\hline & & & & & & & &\\[-1em]
  		$\tilde{d} \pmod{16}$ & Exp. & ${\omega}$ & Exp. & ${\omega}$ &
  		Exp. & ${\omega}$ & Exp. & ${\omega}$ \\ \hline $11$ & $(1,1,0)$ & $3$ &
  		$(1,1,1)$ & $7$ & $(1,3,0)$ & $ 3$ & $(1,3,1)$ & $7$\\ \hline $11$ &
  		$(3,1,0)$ & $3$ & $(3,1,1)$ & $7$ & $(3,3,0)$ & $3$ & $(3,3,1)$ &
  		$7$\\ \hline $15$ & $(1,0,0)$ & $1$ & $(1,0,1)$ & $5$ & $(1,2,0)$ &
  		$1$ & $(1,2,1)$ & $5$\\ \hline $15$ & $(3,0,0)$ & $1$ & $(3,0,1)$ &
  		$5$ & $(3,2,0)$ & $1$ & $(3,2,1)$ & $5$\\ \hline
  	\end{tabular}
  	\caption{Relation between $\epsilon$ and ${\omega}$ for
  		$\tilde{d} \equiv 11,15 \pmod{16}$ \label{table:cased11and15}}
  \end{table}
\end{itemize}

 When $8 \mid d$, Theorem~\ref{thm:epsilondescription}
implies that the case $2\omega = \epsilon \bar{\alpha}^2$ occurs
precisely for $\epsilon \equiv -1 \pmod{\id{p}_2^3}$. Recall that
$(\Om_{\id{p}_2}/2^3)^\times$ is generated by the elements
$\{-1,5,1+\sqrt{d/4}\}$ (of order $2, 2, 8$). Using the congruence of
$\epsilon$ modulo $\id{p}_2^3$, the condition ~(\ref{eq:tilded0}) and
the fact that $2\omega$ is the norm of an element, we search for all
possible values of $\epsilon$ and $\omega$.
\begin{itemize}
\item If $\tilde{d}\equiv 2 \pmod{16}$ (respectively
  $\tilde{d}\equiv 10 \pmod{16}$) then $\#Q_3+\#Q_5$ is even
  (respectively odd). The assumption that $2\omega$ is a norm implies
  that ${\omega}\equiv1,7\pmod{8}$ (respectively
  ${\omega}\equiv3,5\pmod{8}$). All the possible values of $\epsilon$
  for each ${\omega}$ are given in Table~\ref{table:cased2} from
  which it follows (using the definition of $\chi_2$) that
  (\ref{eq:compatibility}) holds.
  \item If $\tilde{d}\equiv 6 \pmod{16}$ then $\#Q_3+\#Q_5$ is odd. The
  norm condition implies that ${\omega}\equiv5,7\pmod{8}$. The
  possible values of $\epsilon$ and ${\omega}$ are given in
  Table~(\ref{table:cased2}) from which it follows that
  (\ref{eq:compatibility}) holds.
\begin{table}[h!]
\begin{tabular}{|c||c|c||c|c||c|c||c|c|}
\hline & & & & & & & &\\[-1em]
$\tilde{d} \pmod{16}$ & $\epsilon$ & ${\omega}$ & $\epsilon$ & ${\omega}$ & $\epsilon$ & ${\omega}$ & $\epsilon$ & ${\omega}$\\
\hline
$2$ & $-1$ & $7$ & $(1+\sqrt{\tilde{d}})^2$ & $1$ & $-(1+\sqrt{\tilde{d}})^4$ & $7$ & $(1+\sqrt{\tilde{d}})^6$ & $1$\\
\hline
 $10$ & $-1$ & $3$ & $(1+\sqrt{\tilde{d}})^2$ & $5$ & $-(1+\sqrt{\tilde{d}})^4$ & $3$ & $(1+\sqrt{\tilde{d}})^6$ & $5$\\
\hline
  $6$& $-1$ & $5$ & $5(1+\sqrt{\tilde{d}})^2$ & $7$ & $-(1+\sqrt{\tilde{d}})^4$ & $5$ & $5(1+\sqrt{\tilde{d}})^6$ & $7$ \\
  \hline
    \end{tabular}
    \caption{Relation between $\epsilon$ and ${\omega}$ for
      $\tilde{d} \equiv 2, 6, 10 \pmod{16}$ \label{table:cased2}}
  \end{table}
\item If $\tilde{d}\equiv14\pmod{16}$ then $\#Q_3+\#Q_5$ is even,
  hence $\chi_2$ is trivial. The norm condition implies that
  ${\omega}\equiv1,3\pmod8$ so formula ~(\ref{eq:compatibility}) holds.
\end{itemize} Once the compatibility is verified, the proof of Theorem
3.2 in \cite{PT} works mutatis mutandis.
\end{proof}
	
\section{The conductor and Nebentypus of the extended
  representation}\label{sec:nebentypus}
	
Let $(a,b,c)$ be a primitive solution of (\ref{eq:24p}) and let $\E$
be the elliptic curve attached to it, with defining equation
(\ref{eq:freycurve}). The properties imposed on $\chi$ imply that the
twisted representation $\rho_{\E,p} \otimes \chi$ extends to a
$2$-dimensional representation of $\Gal_\Q$. 

\begin{lemma} \label{lemma:chi}
Suppose that there exists an odd prime $p$ ramifying in
  $K/\Q$. Let $\sigma \in \Gal_\Q$ and let $\delta_K$ denote the
  quadratic character corresponding to the real quadratic extension
  $K/\Q$. Then,
  \[
\chi(\sigma^2)= \varepsilon(\sigma)\delta_K(\sigma).  
    \]
\end{lemma}
\begin{proof}
  If $\sigma \in \Gal_K$, then the first property of
  Theorem~\ref{thm:charexistence} implies that
  $\chi(\sigma^2) = \chi(\sigma)^2 = \varepsilon(\sigma)$, so the
  statement is clearly true for all elements of $\Gal_K$ (since
  $\delta_K(\sigma) = 1$). Since $\Gal_K$ has index two in $\Gal_\Q$,
  it is enough to prove that the equality holds at one element of
  $\Gal_\Q \setminus \Gal_K$. Let $p$ be an odd prime ramifying in the
  extension $K/\Q$, and let $L = \Q(\zeta_p)$ be the cyclotomic
  extension. The Galois group $\Gal(L/\Q)$ is isomorphic to the cyclic
  group $(\Z/p)^\times$. Let $g$ be a generator. By class field
  theory, $\Gal(L/\Q)$ is also isomorphic to the quotient
  $\II_\Q / \norm_{L/\Q}(\II_L)$. Let $\sigma_p$ be the element of
  $\Gal(L/\Q)$ corresponding to the id\`ele $\iota_p$ with local
  coordinates:
  \[
    (\iota_p)_v=
    \begin{cases}
      g & \text{ if }v = p,\\
      1 & \text{ otherwise.}
    \end{cases}
  \]
  Denote also by $\sigma_p$ any extension to the whole Galois group
  $\Gal_\Q$ of it which is not the identity on $K$. As explained
  before, it is then enough to prove the equality at the element
  $\sigma_p$. Clearly $\sigma_p^2 \in \Gal_K$, and furthermore, it
  matches the transfer map from $\Gal_\Q ^{\text{ab}}$ to
  $\Gal_K^{\text{ab}}$ (see for example \cite[Chapter 8]{MR554237} for
  the definition of the transfer map). On the id\`ele side, the
  transfer map matches the natural map $\II_\Q \to \II_K$, so the
  element $\iota_p$ corresponds to the id\`ele $\iota_p^K$ of $\II_K$
  with local components
  \[
    (\iota_p^K)_v =
    \begin{cases}
      g & \text{ if }v = \id{p},\\
      1 & \text{ otherwise.}
      \end{cases}
    \]
    The value $\chi(\sigma_p^2)$ then equals $\chi(\iota_p^K) = \chi_{\id{p}}(g)$, and
    one of the key properties imposed on $\chi$ and $\varepsilon$ in
    \cite{PT} is that at all odd ramified primes
    $\chi_{\id{p}} = \varepsilon_p \delta_{K,p}$, via the natural
    identification of $(\Z/p)^\times$ with
    $(\Om_K/\id{p})^\times$. Hence the statement.
\end{proof}

\begin{thm} Suppose there exists a prime $q>3$ ramifying in $K$. Then
  the twisted representation $\rho_{\E,p}\otimes \chi$ descends to a
  $2$-dimensional representation of $\Gal_\Q$ attached to a newform of
  weight $2$, Nebentypus $\varepsilon$ and level $N$ given by
  \[
    N=2^e \cdot \prod_{q}q^{v_\id{q}(N({\E}))}\cdot \prod_{q\in Q_3} q\cdot
    \prod_{q \in Q_1 \cup Q_5 \cup Q_7}q^2,
  \]
  where the first product is over odd primes, and $\id{q}$ denotes a prime
  of $K$ dividing $q$. The value of $e$ is one of:
  \[
    e = \begin{cases}
      1,8 & \text{ if } 2 \text{ splits},\\
      8 & \text{ if } 2 \text{ is inert},\\
    7,8 & \text{ if } d \equiv 3 \pmod 8,\\
    5, 8 & \text{ if } d\equiv 7 \pmod 8,\\
      8, 9 & \text{ if }2 \mid d.
    \end{cases}
  \]
  \label{thm:levelandnebentypus}
\end{thm}
	
\begin{proof} The extension result is well known although a proof was
  recalled in \cite[Theorem 4.2]{PT}. To ease notation let
  $\rho'_p=\rho_{\E,p}\otimes \chi$ and $\tilde{\rho}_p$ denote its
  extension to $\Gal_\Q$. The Nebentypus assertion was only proved
  under the hypothesis that $K/\Q$ is imaginary quadratic. The reason
  is the following: we know that $\rho'_p$ has determinant the
  cyclotomic character (denoted $\chic$) times $\varepsilon$ (by Theorem~\ref{thm:charexistence}), hence
  the determinant of $\tilde{\rho}_p$ equals $\varepsilon \chic$ or
  $\varepsilon \delta_K \chic$ (where $\delta_K$ denotes the quadratic
  character corresponding to the extension $K/\Q$). But Ribet's result
  (see \cite[Theorem 4.4]{MR2058653}) implies that the determinant of $\tilde{\rho}_p$
  is odd hence the statement. When $K/\Q$ is real both characters
  take the same value at complex conjugation!  How can we distinguish
  which one is the Nebentypus of the representation
  $\tilde{\rho}_p$ when our extension is real?  The
  solution is to work with another element of an inertia subgroup of
  $K/\Q$.

  Fix a basis for the Tate module of the elliptic curve $\E$ (so we
  can assume that the image of our representation lies in
  $\GL_2(\Q_p)$). Since our field $K$ is real quadratic, we know
  that the Galois representation $\rho_{\E,p}$ is absolutely
  irreducible. In particular, any matrix commuting with its image must
  be a scalar matrix by Schur's Lemma.

  Let $S$ denote the set of primes ramifying in $K/\Q$, and for each
  odd prime $q \in S$ let $\id{q}$ denote the prime of $K$ dividing
  it. Fix one odd prime $q>3$ in $S$ different from  $p$.
  Let $I_q \subset \Gal_\Q$ denote an inertia subgroup at $q$ and
  $I_{\id{q}}$ its index two subgroup.  By \cite[Lemma 2.5]{PT} the
  curve $\E$ has good reduction at $\id{q}$ hence (by the
  N\'eron-Ogg-Shafarevich criterion) $\rho'_p|_{I_{\id{q}}}$ is a
  scalar matrix. Let $\sigma_q \in I_q \setminus I_{\id{q}}$ and let
  $^{\sigma_q}\rho'_p(\tau):=\rho'_p(\sigma_q\tau
  \sigma_q^{-1})$. The character $\chi$ was constructed so that
  $^{\sigma_q}\rho'_p \simeq \rho'_p$, hence both
  representations are conjugate under a matrix of
  $\GL_2(\Q_p)$. Since $\tilde{\rho}_p$ extends $\rho'_p$,
  $\tilde{\rho}_p(\sigma_q)$ is such a matrix. Consider the following
  two different cases:
\begin{itemize}
\item If $^{\sigma_q}\rho'_p= \rho'_p$, then
  $\tilde{\rho}_p(\sigma_q)$ is a scalar matrix (by Schur's
    Lemma), say
  $\left(\begin{smallmatrix} \lambda & 0\\ 0 &
      \lambda\end{smallmatrix}\right)$. In particular,
  $\det(\tilde{\rho}_p(\sigma_q)) = \lambda^2$. On the other hand,
  $\tilde{\rho}_p(\sigma_q)^2 = \rho'_p(\sigma_q^2) =
  \left(\begin{smallmatrix}\chi(\sigma_q^2) & 0\\ 0 &
      \chi(\sigma_q^2)\end{smallmatrix}\right)$ hence in particular
  $\lambda^2 = \chi(\sigma_q^2) =
  \varepsilon(\sigma_q)\delta_K(\sigma_q)$ from Lemma~\ref{lemma:chi},
  so $\det(\tilde{\rho}_p) = \varepsilon \delta_K\chic$.
			
\item If $^{\sigma_q}\rho'_p\neq \rho'_p$,
  $\tilde{\rho}_p(\sigma_q)^2 = \rho'_p(\sigma_q^2)$ is a scalar
  matrix. Then we can chose another basis of the Tate module so
    that the matrix $\tilde{\rho}_p(\sigma_q)$ equals the matrix
    $\left(\begin{smallmatrix}\lambda & 0\\ 0 &
        -\lambda \end{smallmatrix} \right)$. Then
    $\det(\tilde{\rho}_p(\sigma_q)) = -\lambda^2$. Once again,
    $\tilde{\rho}_p(\sigma_q)^2 = \rho'_p(\sigma_q^2) =
    \left(\begin{smallmatrix}\chi(\sigma_q^2) & 0\\ 0 &
        \chi(\sigma_q^2)\end{smallmatrix}\right)$ hence in particular
    Lemma~\ref{lemma:chi} (and the fact that $\delta_K(\sigma_q)=-1$) implies that
    $\det(\tilde{\rho}_p(\sigma_q)) = -\lambda^2 = -\chi(\sigma_q^2)=
    -\varepsilon(\sigma_q)\delta_K(\sigma_q) = \varepsilon(\sigma_q)$ so 
  $\det(\tilde{\rho}_p) = \varepsilon \chic$.
\end{itemize}
Then we are left to prove that $^{\sigma_q}\rho'_p \neq \rho'_p$ (a
result independent of the prime $q \in S$). Recall that
$\rho'_p= \rho_{\E,p} \otimes \chi$, hence the statement is equivalent to
prove that
$^{\sigma_q} \rho_{\E,p} \neq \rho_{\E,p} \cdot \delta_{-2}$ (since
$^{\sigma_q}\chi = \chi \delta_{-2}$). Consider both actions for $\tau \in \Gal_K$ on
points of $\E$ of order $p^n$: the left hand side equals
$\sigma_q \cdot \tau \cdot \sigma_q^{-1}(P)$, while the right hand side
equals $\delta_{-2}(\tau)\tau(P)$.

Consider the $2$-isogeny $\phi : \E \to \overline{\E}$ explicitly
given by
\[
  \phi(x,y) = (\phi_1(x,y),\phi_2(x,y))=
  \left(\frac{-y^2}{2x^2},\frac{y(2a^2+2\sqrt{d}b-x^2)}{2\sqrt{-2}x^2}\right).
\]
Note in particular that
\begin{equation}
  \label{eq:isogenydefinition}
  \delta_{-2}(\tau)\cdot \tau \circ \phi = \phi \circ \tau \text{ for all }\tau \in \Gal_K,
\end{equation}
where we consider $\delta_{-2}(\tau)$ as an
endomorphism of $\E$. The hypothesis on $p$ being odd implies that for all positive integers $n$, the map $\phi : \E[p^n] \to \overline{\E}[p^n]$ is bijective. Then if $P \in \E[p^n]$, we have
\[
\sigma_q\cdot \tau \cdot\sigma_q^{-1}(P)= (\sigma_q \cdot \phi^{-1})(\phi\tau \phi^{-1})(\sigma_q\cdot \phi^{-1})^{-1}(P) = \delta_{-2}(\tau) (\sigma_q\cdot \phi^{-1})\tau (\sigma_q \cdot \phi^{-1})^{-1}(P),
\]
where the last equality follows
from~(\ref{eq:isogenydefinition}). Take $n$ large enough so that the
representation on $p^n$-torsion points (that we denote $\rho_n$) is
absolutely irreducible.
Then by Schur's Lemma, $^{\sigma_q}\rho_n = \rho_n \cdot \delta_{-2}$
if and only if the endomorphism $\sigma_q\phi^{-1}$ acts as a scalar
matrix on $\E[p^n]$. Since the Galois representation of an elliptic
curve is a part of a compatible family (and the Nebentypus does not
depend on the choice of the prime $p$), it is enough to consider the case
$p=3$ and prove that $\sigma_q\phi^{-1}$ acting on the $3$-torsion
points is not equal to multiplication by $\pm 1$ (then it cannot act
as multiplication by an integer on points of order $3^n$).
  
%
%
Note that $-1$ acts trivially on the $x$-coordinates of torsion
points, hence it is enough to prove that on the $x$-coordinate of the
$3$-torsion points, the elements $\sigma_q$ and $\phi$ do not
coincide. Let $M = K(x(\E[3]))$ denote the extension of $K$ obtained
by adding to $K$ the $x$-coordinates of all points in $\E[3]$ (a
degree $2$ subextension of $K(\E[3])$). Note on the one hand that
$\phi$ maps $x$-coordinates of $3$-torsion points of $\E$ to $x$-coordinates of $3$-torsion points of $\overline{\E}$, but also, the
map $\phi_1$ is given by a polynomial in $x$ with coordinates in
$K$. More concretely,
\begin{equation}
  \label{eq:phi1defi}
\phi_1(x) = -\frac{x^3+4ax^2+2(a^2+\sqrt{d}b)x}{2x^2}.  
\end{equation}
    This implies that $M$ is a Galois extension of $\Q$. Clearly, both $K$ and
    $\Q(\sqrt{-3})$ are subfields of $M$ (since the determinant of our
    representation is the cyclotomic character modulo $3$). In
    particular, $\Q(\sqrt{-3d})$ is contained in $M$.  Since the
    ramification degree of $q$ in $M/\Q$ is two (because $\E$ has
    good reduction at the prime dividing $q$), it must be the case
    that $\Q(\sqrt{-3}) \subset M^{\sigma_q}$ (since $\sigma_q$ cannot
    fix $\sqrt{d}$ nor $\sqrt{-3d}$). 

    For a generic curve $y^2 = x^3+\alpha x^2 + \beta x$, its
    $3$-division polynomial (whose roots generate the extension $M/K$)
    is given by
\begin{equation}
  \label{eq:divpol} \psi_3(x)=
  3x^4+4\alpha x^3+6\beta x^2-\beta^2.
\end{equation}
(recall that in our case, $\alpha = 4a$ while $\beta = 2(a^2+\sqrt{d}b)$).
Let $\theta_1, \ldots, \theta_4$ be the roots of $\psi_3$ and let
$\bar{\beta} = \frac{\alpha^2-4\beta}{4}$ (in our case $\bar{\beta}$ matches the
conjugate of $\beta$). Then
\begin{equation}
  \label{eq:disc} \frac{\Disc(\psi_3)}{2^{12}\cdot
    3^2\cdot \beta^4 \cdot \bar{\beta}^2}=\left(\frac{\prod_{i < j}(\theta_i -
      \theta_j)}{2^6\cdot 3 \cdot \beta^2 \cdot \bar{\beta}}\right)^2 = -3.
\end{equation}
In particular, since $\sigma_q$ fixes $\sqrt{-3}$, it must fix the
quotient
$\frac{\prod_{i < j}(\theta_i - \theta_j)}{2^6\cdot 3 \cdot \beta^2
  \cdot \bar{\beta}}$, and since $\sigma_q$ is not the identity in
$K$, it must send $\beta$ to $\bar{\beta}$ and vice-versa. In particular,
\[
\sigma_p\left(\prod_{i<j}(\theta_i - \theta_j)\right)=\prod_{i<j}(\sigma_p(\theta_i) - \sigma_p(\theta_j)) = \prod_{i<j}(\theta_i - \theta_j) \cdot \frac{\bar{\beta}}{\beta}.
  \]
On the other hand, for $i \neq j$, using~(\ref{eq:phi1defi}) we get
\[
  \phi_1(\theta_i) - \phi_1(\theta_j) = -\frac{\theta_i^2+\alpha \theta_i + \beta}{2\theta_i} + \frac{\theta_j^2+\alpha \theta_j + \beta}{2\theta_j} = (-1)(\theta_i-\theta_j)\frac{(\theta_i \theta_j-\beta)}{2\theta_i\theta_j}.
  \]
  It is not hard to verify that if $\{\theta_1,\ldots ,\theta_4\}$ are
  roots of a monic polynomial $x^4+A_1x^3+A_2x^2+A_3x+A_4$, then
\[
  \prod_{i<j}(\theta_i \theta_j -\beta) =\beta^6-A_2\beta^5+(A_1A_3-A_4)\beta^4+(2A_4A_2-A_4A_1^2-A_3^2)\beta^3 + (A_4A_3A_1-A_4^2)\beta^2 -A_4^2A_2\beta + A_4^3.
\]
Using this formula for $\psi_3$, we obtain
\[
  \prod_{i<j}(\theta_i \theta_j -\beta)= \frac{16\beta^5}{27}(\alpha^2-4\beta) = \frac{64 \beta^5 \bar{\beta}}{27}.
\]
Then
\[
\prod_{i < j} (\phi_1(\theta_i)-\phi_1(\theta_j)) = (-1)^6\prod_{i<j} (\theta_i -\theta_j) \left(\frac{-\bar{\beta}}{\beta}\right).
  \]
In particular, the action of $\phi_1$ and $\sigma_q$ do not match in the roots $\theta_i$ so  the claim follows.
\end{proof}
	
\begin{remark} \label{rem:sqrt6} The same result holds for
  $K = \Q(\sqrt{3})$ or $\Q(\sqrt{6})$ replacing the $3$-torsion
  points computation with the $5$-torsion ones (for the prime
  $q=3 \in S$). While working with $5$-torsion points, formula
  (\ref{eq:disc}) becomes
  \[
    \frac{\Delta(\psi_5)}{2^{88}\cdot 5^{10}\cdot
b^{44}\cdot (a^2-4b)^{22}} = 5.
\]
The case $K=\Q(\sqrt{2})$ is more subtle as there
is no clear choice of an order two element in the Galois group
$\Gal(K(\E[p])/\Q)$. In particular computed examples the result
holds (but we do not have a general proof).
\end{remark}

\section{Ellenberg's result}\label{section:Ellenberg}
Let $K/\Q$ be a quadratic extension, and let $E/K$ be a $\Q$-curve
$2$-isogenous to its Galois conjugate with a prime $\ell >3$ of
potentially multiplicative reduction. Then following ideas of
Darmon-Merel, Ellenberg proved (in Propositions 3.2, 3.4, 3.14 and
Section 4 of \cite{MR2075481}) that the projective modulo $p$
representation of $E$ is surjective if either:
\begin{itemize}
\item there exists $f \in S_2(2p^2)$ such that $w_p f= f$ and
  $w_2 f = -f$, or
\item there exists $f \in S_2(p^2)$ such that $w_p f = f$,
\end{itemize}
with $L(f \otimes \delta_K,1) \neq 0$. Recall here that if
$f = \sum_n a_n q^n$ is a modular form and $\psi$ is a Dirichlet character,
then $f \otimes \psi$ denotes the newform attached to the modular form
$\sum_n a_n \psi(n)q^n$.

An important result of
Ellenberg (see \cite[Proposition 3.9]{MR2075481}) proves that if $K$ is an
imaginary quadratic field then there is always a modular form
satisfying the second hypothesis for $p$ large enough.
	
\begin{prop}
  \label{prop:2inert}
  If $K/\Q$ is a real quadratic field in which $p$ is unramified,
  then there does not exist a newform satisfying any of the two previous
  conditions unless $2$ splits in $K/\Q$.
\end{prop}
\begin{proof} For a newform $f$, let $\epsilon(f)$ denote its root
  number (i.e. the sign of the functional equation). Recall from
  \cite[\S I.5]{MR1431508} that if $f \in S_2(N)$ is a
  newform and $\psi$ is a Dirichlet character whose conductor is prime
  to $N$ then $\epsilon(f \otimes \psi) = \epsilon(f)
  \psi(-N)$. Suppose that $f \in S_2(p^2)$ satisfies that
  $w_pf = f$, so its root number equals $-1$ (recall that the root number equals minus the sign of the canonical involution). Then if $p$ is
  unramified in $K/\Q$, the twisted form $f \otimes \delta_K$ has also
  root number $-1$ (since $\delta_K(-p^2) = 1$ for $K$ real quadratic),
  so $L(f\otimes \delta_K,1)=0$.
		
Suppose that $f$ is a newform of level $2p^2$. The Atkin-Lehner
eigenvalues hypotheses imply that $\epsilon(f)=1$. Suppose that $2$ is
unramified in $K/\Q$, hence
$\epsilon(f \otimes \delta_K) = \delta_K(-2p^2) = \delta_K(2) = 1$ if
and only if $2$ splits in $K/\Q$. When $2$ ramifies in $K/\Q$, we can
write $d_K = d_1 \cdot d_2$, where $d_1 \in \{-4, \pm 8\}$ and $d_2$
is an odd fundamental discriminant. Suppose $d_1 = -4$; writing
$f \otimes \delta_K = (f \otimes \delta_{d_1})\otimes \delta_{d_2}$,
it is enough to understand the sign change for the first twist (the
form $f \otimes \delta_{-4}$ being a form of level $16p^2$). By a
result of Atkin-Lehner (see \cite[Theorem 7]{MR268123})
$w_2(f \otimes \delta_{-4})=-1$ while
$w_p(f \otimes \delta_{-4})= w_p(f)$, hence
$\epsilon(f \otimes \delta_{-4}) = \epsilon(f) = 1$ and since $d_2$ is
negative (hence $\delta_{d_2}(-1) = -1$)
$\epsilon(f \otimes \delta_K)=-1$. A similar computation (using that
$w_2(f \otimes \delta_8) = 1$ and $w_2(f \otimes \delta_{-8})=-1$)
proves the remaining cases.
\end{proof}
	
Suppose then that $2$ splits in $K/\Q$. Ellenberg's proof of the
existence of a newform with prescribed properties consists on bounding
an average of twisted central values in the whole space of level $p^2$
modular forms (since the forms with the wrong Atkin-Lehner involution
sign in such space have zero central value).  While considering the
space $S_2(2p^2)^{\text{new}}$ the computations are harder,
as one needs to compute an average not over the whole space, but over
the subspace with a chosen Atkin-Lehner sign at $p$ (therefore
imposing also a condition to the Atkin-Lehner sign at $2$). Such
computation was carried out in \cite{LeFourn} (see the proof of
Corollary 4). Unfortunately, explicit constants are not presented in
Le Fourn's article, hence we need to add some (minor) extra details to
its proof (we suggest the reader to have a copy of such article in
hand for the rest of this section as we follow its notations and
definitions; specially Section 6).
	
The inequality $J_1(x) \le \frac{|x|}{2}$ and
$|S(1,n;c)|< \sqrt{c}\tau(c)$ (used in Ellenberg's article) turns
inequality $(6.3)$ of \cite{LeFourn} into
\begin{equation}
  \label{eq:1} |A_{N,Q,c}(x)| \le \frac{\pi}{3} \cdot
  \frac{xe^{-2\pi/x}\tau(c)}{Qc^{3/2}},
\end{equation} for $x \ge 71$ (using that
$(1-e^{-2\pi/x})^{-1} \le \frac{x}{6}$ when $x \ge 71$). The same
bound for $J_1$ gives the explicit inequality for equation $(6.4)$
\begin{equation}
  \label{eq:2} |A_{N,Q,c}(x)| \le
  \frac{12}{\pi}\frac{(\log(Dc)+1)\sqrt{D}}{cQ} e^{-2\pi/x}.
\end{equation} To get a bound for $A_{N,Q}(x)=2\pi \sum_{c>0,
  (N/Q)\mid c, (c,Q)=1}A_{N,Q,c}(x)$ we split the sum as in
\cite{LeFourn}. Suppose that $N \neq Q$, so in the following sum there is no term for $c=D$:
\[
  \frac{|A_{N,Q}(x)|}{2\pi} \le
\frac{12}{\pi}\frac{\sqrt{D}e^{-2\pi/x}}{Q} \sum_{\stackrel{c <
x^2}{(N/Q)\mid c}} \frac{(\log(Dc)+1)}{c}+ \frac{\pi}{3}
\sum_{\stackrel{ c > x^2}{(N/Q)\mid c}} \frac{x
e^{-2\pi/x}\tau(c)}{Qc^{3/2}}.
\]
For the first inner sum, writing $c = (N/Q)b$, we get the
inequality
\begin{multline}
  \label{eq:3}
  \sum_{\stackrel{c < x^2}{(N/Q)\mid c}}
  \frac{(\log(Dc)+1)}{c} =
 \frac{Q}{N}\left(\left(1+\log(\frac{DN}{Q})\right)\sum_{b=1}^{\frac{x^2Q}{N}}\frac{1}{b} + \sum_{b=1}^{\frac{x^2Q}{N}}\frac{\log(b)}{b}\right) \le \\
\frac{Q}{N}\left(\left(1+\log(\frac{DN}{Q})\right)\left(1+\log(\frac{x^2N}{Q})\right) +
\frac{\log^2(\frac{x^2N}{Q})}{2}\right),
\end{multline}
where the last inequality comes from the usual comparison between the
series and the integral. To bound the sum
$\sum_{c> X^2}\frac{\tau(c)}{c^{3/2}}$, recall the following
inequalities:
\begin{enumerate}
\item For all real $s > 1$, $\sum_{n \ge X}\frac{1}{n^s} \le -\frac{X^{1-s}}{1-s} +
  \frac{X^{-s}}{2}$ (see for example \cite[Lemma 3.1]{2003.05887}),
  
\item For $X>1$ a real number, $\sum_{d \le X}\frac{1}{d} \le \log(X) +
\gamma + \frac{7}{12X}$ where $\gamma$ is the Euler-Mascheroni constant, $\gamma
\le 0.58$ (see equation (3.1) of \cite{MR937932}).
\end{enumerate}
Then, if $s>1$,
\begin{multline*}
  \sum_{n\ge X} \frac{\tau(n)}{n^s} = \sum_{n \ge X} \left(\sum_{d|n}
    \frac{1}{n^s}\right) = \sum_d \frac{1}{d^s} \sum_{m \ge X/d}
  \frac{1}{m^s} \le \zeta(s)\sum_{d>X}\frac{1}{d^s} + \sum_{d \le X}
  \frac{1}{d^s}\left(-\frac{(X/d)^{1-s}}{(1-s)} +
    \frac{(X/d)^{-s}}{2}\right) \le \\ \zeta(s)\left(
    -\frac{X^{1-s}}{(1-s)} + \frac{X^{-s}}{2}\right) -
  \frac{X^{1-s}}{(1-s)} \sum_{d\le X} \frac{1}{d} + \frac{X^{1-s}}{2}
  \le \\ \zeta(s)\left( -\frac{X^{1-s}}{(1-s)} +
    \frac{X^{-s}}{2}\right) - \frac{X^{1-s}}{(1-s)} (\log(X)+\gamma
  +\frac{7}{12X}) + \frac{X^{1-s}}{2}.
\end{multline*}
Substituting at $s=3/2$, $X$ by $X^2$ and assuming $X \ge 32$, we
obtain
\begin{equation}
		\label{eq:secondbound} \sum_{n \ge
X^2}\frac{\tau(n)}{n^{3/2}} \le \frac{6 \log(X)}{X}.
\end{equation}
Using both inequalities, we get (for $N \neq Q$) 
\begin{multline}
  \frac{|A_{N,Q}(x)|}{2\pi} \le
  \frac{12\sqrt{D}e^{-2\pi/x}}{N\pi}\left(\left(\log(\frac{DN}{Q})+1\right)\left(1+\log(\frac{x^2N}{Q})\right)
    + \frac{\log^2(\frac{x^2N}{Q})}{2}\right) +\\ +\frac{2\pi}{N}
  \sqrt{Q/N}\tau(N/Q)\log(x)e^{-2\pi/x}.
\label{eq:boundA}
\end{multline}
When $N = Q$, there is an extra term $\frac{\pi}{3} \frac{x e^{\frac{-2\pi}{x}}\tau(D)}{ND^{3/2}}$ corresponding to the value $c=D$.
Using the fact that $B_{N,Q}(x) = A_{N,Q}(D^2N/x)$, we get the bound
\begin{equation}
  \label{eq:5}
  \frac{|B_{N,Q}(x)|}{2\pi} \le \frac{|A_{N,Q}(D^2N/x)|}{2\pi}+
  \delta_{Q = N}\frac{\pi}{3}\frac{\sqrt{D}}{x}\tau(D)e^{\frac{-2\pi
      x}{ND^2}}.
\end{equation}
Recall that
$(a_1,L_\chi)_{2p^2}^{+_{p^2},\text{new}}=(a_1,L_\chi)_{2p^2}^{+_{p^2}}
- \frac{1}{p-1}(a_1,L_\chi)_{2p}^{\chi(p)_p}$ (see \cite[Lemma
4.1]{LeFourn}), hence formulas~$(6.1), (6.2)$ of \cite{LeFourn} give
	
\begin{multline}
\frac{1}{2\pi}(a_1,L_\chi)_{2p^2}^{+_{p^2},\text{new}} \ge
\frac{(p-2)}{(p-1)}e^{-2\pi/x} - \left(|A_{2p^2,1}(x)| + |A_{2p^2,p^2}(x)|
+ \frac{|A_{2p,1}(x)|}{p-1} + \frac{|A_{2p,p}(x)|}{p-1}+\right .\\ \left. +|B_{2p^2,2p^2}(x)| +
|B_{2p^2,2}(x)| + \frac{|B_{2p,2p}(x)|}{p-1} + \frac{|B_{2p,2}(x)|}{p-1}\right).
\label{eq:final}
\end{multline}
Taking $x$ of the same magnitude of $p$ (in our applications we will
take $x = p\cdot \kappa$ for a numerical computed constant $\kappa$),
the right hand side is an increasing function of $p$, hence as soon as
we find a positive value for it, we get an explicit bound.
	
\section{Examples}
\label{sec:examples}
In this section, instead of working with fundamental discriminants, we
take values of $d$ which are square-free. We applied the method to
study solutions of (\ref{eq:24p}) for square-free values
$1 \le d \le 20$ and $d=129$. The field $\Q(\sqrt{6})$ is the first
one where the fundamental unit has norm $1$ and also contains a
non-trivial solution for all primes $p$. The case $d=129$ is the first
field where $2$ splits (so Ellenberg's result can be applied) and also
where all newforms could be discarded using Mazur's trick. For
$d\in\{3,5,7,14\}$ there are modular forms without complex
multiplication that cannot be discarded with the aforementioned
strategy (so the modular method fails). For the other square-free
values of $d$, the modular method does give a positive answer but only
for primes $p>M$ (an explicit constant) with a prescribed congruence
condition. A summary of the results is presented in
Table~\ref{table:examplesd}. The table contains also the dimension of
the weight two newform space (computed to discard possible solutions)
as well as the dimension of the Hilbert parallel weight $2$ modular
forms space (if one would follow the classical modular approach over
$K$). Note the dimension of the Hilbert space becomes almost
infeasible from a computational point of view very soon.

\begin{table}[h!]
  \begin{tabular}{|c||c|c|c|c|c|}
    \hline
    $d$ & Theorem & $M$ & Condition on $p$ & $\dim(S_2(N,\varepsilon))$ & Hilbert space\\
    \hline
    $6$ & \ref{thm:d=6} & $19$ & $p\neq 97$; $p\equiv1,3\pmod8$ & $28$, $64$ & $96 $, $384$ \\
    \hline
    $10$ & \ref{thm:d=10} &$19$ & $p\neq 139$; $p\equiv1,3\pmod8$ & $140, 288$
                                                                    & $448$, $1792$ \\
    \hline
    $11$ & \ref{thm:d=11} & $19$ & $p\neq73$;
                                   $p\equiv1,3\pmod8$ & $48, 92$ & $224$, $896$ \\
    \hline
    $19$&  \ref{thm:d=19} & $19$ & $p\neq 43,113$; $p\equiv1,3\pmod8$ & $80, 156$ & $608$, $2432$
    \\ \hline
    $129$ & \ref{thm:cased=129} & $19$ & $p>900$ or $p\equiv 1,3\pmod 8$ and
    $p\neq43$ & $16$, $1400$ & $100$, $600$, $38400$\\ \hline
  \end{tabular}
  \caption{\label{table:examplesd}}
\end{table}
	\subsection{The case $d=6$} As mentioned before, although the
case $d=6$ seems to be out of reach of the modular method, it turns
out that the Frey curve attached to the solution $(\pm 7, \pm 20, 1)$
does also have complex multiplication! (this seems like a very
fortunate coincidence, unlikely to occur for other values). The
trivial solution gives an elliptic curve with $j$-invariant $8000$
(with complex multiplication by $\Z[\sqrt{-2}]$). Over $\Q(\sqrt{6})$ there are only two
extra isomorphism classes of elliptic curves with complex multiplication whose
$j$-invariant is not rational (see \cite{MR3373247}), with
$j$-invariants $188837384000\pm77092288000\sqrt{6}$. The Frey curves
$E_{(\pm 7, \pm 20, 1)}$ have precisely such $j$-invariants!
\begin{thm}
  \label{thm:d=6}
  Let $p>19$ be a prime number such that $p\neq 97$ and
  $p\equiv1,3\pmod8$. Then, $(\pm 7,\pm20,1)$ are the only non-trivial
  primitive solutions of the equation
  \[
    x^4-6y^2=z^p.
  \]
\end{thm}
\begin{proof} Suppose that $(a,b,c)$ is a non-trivial
  primitive solution. If $c = \pm 1$ then, by~(\ref{eq:solutionsc=1}),
    $(a,b,c)=(\pm 7,\pm20,1)$. Hence, we are led to consider the case
  $c \neq \pm 1$ (in particular $c$ is divisible by a prime number
  greater than $3$). In order to apply Ribet's lowering the level
  result, we need to prove that the residual representation of $\E$
  modulo $p$ is absolutely irreducible. For that purpose we apply
  Theorem 1 of \cite{MR3346965}. Let $\epsilon = 5+2\sqrt{6}$ be a
  fundamental unit. 
The primes dividing $\lcm(\norm(\epsilon^{12}-1),\norm(\overline{\epsilon}^{12}-1))$ live in $\{2,3,5,11,97\}$. Next we need to compute the
characteristic polynomial at a prime of good reduction. Since
$\E$ has good reduction at primes ramifying in $K/\Q$, $q = 3$
is a good candidate so let $\id{q} = \langle 3 + \sqrt{6}\rangle$. The
curve $\E$ modulo $\id{q}$ is one of $y^2=x^3 \pm x^2+2x$,
hence $a_{\id{q}}(E)= \pm 2$. The resultant between $x^2\pm 2x +3$ and
$x^{12}-1$ is only divisible by the primes $\{2, 3, 19, 97\}$, hence
the residual image is absolutely irreducible for all primes except the
ones in the set $\{2, 3, 5, 11, 19, 97\}$.  Using
Theorem~\ref{thm:levelandnebentypus} (and Remark~\ref{rem:sqrt6}) and
Ribet's lowering the level result, we have to compute the spaces
$S_2(2^8\cdot 3,\varepsilon)$ and $S_2(2^9 \cdot 3,\varepsilon)$,
where $\varepsilon$ is the character corresponding to the quadratic field
$\Q(\sqrt{3})$. 

\vspace{1pt}
\noindent $\bullet$ The space $S_2(2^8\cdot 3,\varepsilon)$ has $10$
Galois conjugacy classes, $6$ of them having complex
multiplication. Running Mazur's trick (see \cite[Proposition 6.1]{PT})
for primes $5\le q\le 10$ we can discard all newforms except three 
with complex multiplication, if $p\not \in\{2,5,7\}$. The only newforms
that cannot be discarded in this space are the three newforms
corresponding to the solutions $(\pm1,0,1)$ and $(\pm7,\pm20,1)$ with
complex multiplication by $\Z[\sqrt{-2}]$.

\vspace{1pt}

\noindent $\bullet$ The space $S_2(2^9\cdot 3,\varepsilon)$ has $13$
Galois conjugacy classes, $3$ of them having complex
multiplication. Again, running Mazur's trick for primes
$5\le q \le 20$ allows to discard all such newforms if $p\not \in\{2,3,5,7, 17\}$.

\vspace{1pt} Then, assuming $p>19$ and $p\neq 97$ we are able to lower
the level and discard all the possibles newforms except three 
with complex multiplication by $\Z[\sqrt{-2}]$. To discard the
remaining ones we need to impose a congruence condition on $p$. If
$p\equiv1,3\pmod8$, then it splits in $\QQ(\sqrt{-2})$ and then the
residual representations of the newforms with complex multiplication
modulo $p$ have image lying in the normalizer of a split Cartan
subgroup. This contradicts \cite[Proposition 3.4]{MR2075481} (as $c$
is divisible by a prime greater than $3$).
	\end{proof}
	\begin{remark}\label{rem:FreitasSiksek} While proving large image, \cite[Theorem
1]{MR3346965} was used with $q=3$, since we know that the curve has
good reduction for odd primes ramifying in $K$. Although we do not
know a priori other primes of good reduction, if the obtained bound is
large not everything is lost. Let $q>5$ be a prime inert in $K$ and
suppose $p>71$. If $q$ divides $c$, the curve has multiplicative
reduction at $q$ hence \cite[Theorem 1.2]{2004.07611} implies that the
residual representation is irreducible. Otherwise, the curve has good
reduction at $q$ hence we can apply the above strategy to the prime
$q$.  This method was used for $d\in\{10,11,19\}$.
	\end{remark}

\subsection{The case $d=10$} In this case we have the following result.

\begin{thm}
  Let $p>19$ be a prime number such that $p\neq 139$ and
  $p\equiv1,3\pmod8$. Then, there are no non-trivial primitive
  solutions of the equation
  \[
    x^4-10y^2=z^p.
  \]
\label{thm:d=10}
\end{thm}
\begin{proof}
  Let $(a,b,c)$ be a putative non-trivial primitive solution. In this
  case, Theorem \ref{thm:charexistence} implies that $\varepsilon$ is
  a character of order $4$ and conductor $4\cdot5$, while $\chi$ has
  order $8$. As in the previous case, applying \cite[Theorem
  1]{MR3346965} and Remark~\ref{rem:FreitasSiksek} for primes $q=5,7$,
  we get that $\overline{\rho_{\E,p}}$ is irreducible if $p$ does not belong to
  $\{2,3,5,7,13,31,37\}$. Hence, by
  Theorem~\ref{thm:levelandnebentypus} and Ribet's lowering the level
  result, we have that there exists a newform $g$ in
  $S_2(2^8\cdot5^2,\varepsilon)$ or in $S_2(2^9\cdot5^2,\varepsilon)$
  whose Galois representation is congruent modulo $p$ to
  $\rho_{\E,p}\otimes\chi$.  \vspace{1pt}
	
  \noindent $\bullet$ The space $S_2(2^8\cdot5^2,\varepsilon)$ has
  $55$ Galois conjugacy classes, $22$ of them having complex
  multiplication. Running Mazur's trick for all the newforms $g$ and
  primes $3\le q \le 37$ such that $q\neq5,31$, we obtain that all
  newforms can be discarded if $p\not \in\{2,3,5,7,11,17,19,23\}$ except
  for the two newforms coming from the trivial solutions, with complex multiplication by
  $\ZZ[\sqrt{-2}]$.
	
  \vspace{1pt}
	
  \noindent $\bullet$ The space $S_2(2^9\cdot5^2,\varepsilon)$ has
  $40$ newforms, $10$ of them having complex multiplication. In this
  case Mazur's trick for primes $q\neq5$ such that $3\leq q\le 20$,
  discards all the newforms in the space if
  $p\not \in\{2,3,5,7,11,13,17,23\}$.
	
  \vspace{1pt} Hence, assuming
  $p\notin\{2,5,7,11,13,17,19,23,31,37\}$, it only remains to discard
  the two newforms with complex multiplication belonging to the first
  space. Since the solution is primitive, $c$ is odd (see
  \cite[Lemma 2.4]{PT}). If $c$ is divisible by $3$, then we can use
  Mazur's trick with $q=3$, getting that
  $p\mid \norm(16\varepsilon^{-1}(3)-a_3(g)^2)$ (see the last line of the
  definition of $B(g,q;a,b,c)$), so $p\in\{2,5\}$. Hence $c$ is
  not divisible by $3$ and we are in the hypothesis of
  \cite[Proposition 3.4]{MR2075481}. Then, once again, we can discard
  the remaining two newforms when $p\equiv1,3\pmod8$.
\end{proof}
\subsection{The case $d=11$} In this case we have the following result.

\begin{thm}
  Let $p>19$ be a prime number such that $p\neq 73$ and
  $p\equiv1,3\pmod8$. Then, there are no non-trivial primitive
  solutions of the equation
  \[
    x^4-11y^2=z^p.
  \] 
  \label{thm:d=11}
\end{thm}
\begin{proof}
  Let $(a,b,c)$ be a non-trivial primitive solution. By
  Theorem~\ref{thm:charexistence} we have that $\varepsilon$ is of
  order $2$ and conductor $4\cdot 11$, and $\chi$ is of order $4$. Applying
  \cite[Theorem 1]{MR3346965} (and again using the strategy of
  Remark~\ref{rem:FreitasSiksek}) for primes $q=11,13$ we get that if
  $p$ does not belong to $\{2,3,5,7,11,17,19,73,397\}$ then
  $\overline{\rho_{\E,p}}$ is absolutely irreducible and we can apply Ribet's
  lowering the level result, so Theorem~\ref{thm:levelandnebentypus}
  implies the existence of a newform $g$ in
  $S_2(2^7\cdot 11,\varepsilon)$ or in $S_2(2^8\cdot 11,\varepsilon)$
  congruent modulo $p$ to $\rho_{\E,p}\otimes\chi$.
	
\vspace{1pt}
	
\noindent $\bullet$ The space $S_2(2^7\cdot11,\varepsilon)$ has $4$
Galois conjugacy classes, none of them with complex
multiplication. Running Mazur's trick for primes $3\le q \le 10$ we
can discard all the newforms if $p>7$.
	
\vspace{1pt}
		
\noindent $\bullet$ The space $S_2(2^8\cdot11,\varepsilon)$ has $15$
Galois conjugacy classes, $7$ of them having complex
multiplication. Two of the newforms with complex multiplication
correspond to the trivial solutions $(\pm1,0,1)$. Running Mazur's
trick for the other $13$ newforms, for primes $q\neq11$ such that
$3\le q\le 43$, we can discard them if $p>19$. To discard
the remaining two newforms we need the hypothesis $p\equiv1,3\pmod8$
and use \cite[Proposition 3.4]{MR2075481}.
\end{proof}

\subsection{The case $d=19$} In this case we have the following
result.

\begin{thm}
  \label{thm:d=19}
  Let $p>19$ be a prime number such that $p\neq 43,113$ and
  $p\equiv1,3\pmod8$. Then, there are no non-trivial primitive
  solutions of the equation
  \[
    x^4-19y^2=z^p.
  \]
\end{thm}
\begin{proof}
  Let $(a,b,c)$ be a non-trivial primitive solution. To prove that the
  residual representation of $\E$ modulo $p$ is absolutely irreducible
  we apply \cite[Theorem 1]{MR3346965} for $q=19$ and follow
  Remark~\ref{rem:FreitasSiksek} for the prime $q=7$, obtaining that
  $\rho_{\E,p}$ has absolutely irreducible reduction if
  $p\notin\{2, 3, 5, 11, 13, 17, 19, 31, 43, 113,$  $115597\}$, so we are
  going to assume this hypothesis from now on.
	
  The character $\varepsilon$ has order two and conductor $4\cdot19$,
  while $\chi$ is of order $4$. Then Ribet's lowering the level result
  together with Theorem~\ref{thm:levelandnebentypus} imply that we
  have to search for a newform $g$ in one of the spaces
  $S_2(2^7\cdot19,\varepsilon)$ or $S_2(2^8\cdot19,\varepsilon)$.
	
  \vspace{1pt}
	
  \noindent $\bullet$ The space $S_2(2^7\cdot19,\varepsilon)$ has $4$
  Galois conjugacy classes, none of them with complex
  multiplication. Using Mazur's trick with primes $3\le q \le 17$ we
  are able to discard all newforms (in fact we just need $p>2$).
	
  \vspace{1pt}
	
  \noindent $\bullet$ The space $S_2(2^8\cdot19,\varepsilon)$ has $18$
  Galois conjugacy classes, $7$ of them having complex
  multiplication. With the above assumption on $p$ (and in fact just
  assuming $p>19$), we can use Mazur's trick with primes
  $3\le q \le 17$ and discard all newforms but two of them,
  corresponding to the trivial solutions (and having complex
  multiplication by $\ZZ[\sqrt{-2}])$.

  To discard these two newforms with complex multiplication, we proceed as before. Since
  the solution is primitive, $c$ must be odd. Suppose that $c$ is
  divisible by $3$. Then, the fact that
  $p\mid \norm(16\varepsilon(3)^{-1}-a_3(g)^2)$ implies that
  $p\in\{2,3\}$, which gives a contradiction. Hence $c$ is not
  divisible by $3$ and then we are in the hypothesis of
  \cite[Proposition 3.4]{MR2075481}, so we can discard the newforms
  attached to the trivial solutions under the assumption $p\equiv1,3\pmod8$.
\end{proof}
\subsection{The case $d=129$} The prime $2$ splits in
$\Q(\sqrt{129})$, hence Ellenberg's result (as described in
Section~\ref{section:Ellenberg}) can be applied to discard the trivial
solutions as well.
	
\begin{thm} Let $p>19$ be a prime number satisfying that either
  $p>900$ or $p\equiv 1,3\pmod 8$ and $p\neq43$. Then, there are no
  non-trivial primitive solutions of the equation
  \[
    x^4-129y^2=z^p.
  \]
  \label{thm:cased=129}
\end{thm}
\begin{proof} As before, let $(a,b,c)$ be a non-trivial primitive
  solution, and $\E$ the Frey curve attached to it. \cite[Theorem
  1]{MR3346965} proves that the residual image is absolutely
  irreducible for primes not in $\{2, 3, 5, 7, 11, 13, 17,$
  $ 43, 53, 251,313, 661, 2593, 3371, 411577\}$. As this bound is a
  little large, we follow the strategy described in \cite[Lemma
  3.2]{2102.11699}. Suppose that the residual extended representation
  $\tilde{\rho}_p$ at a prime $p$ is reducible, say its
  semisimplification is given by $\theta_1 \oplus \theta_2$. Then the
  residual representation of $\rho_{\E,p}$ is isomorphic to
  $\chi^{-1} \theta_1|_{\Gal_K} \oplus \chi^{-1}
  \theta_2|_{\Gal_K}$. To ease notation, let
  $\psi_i = \chi^{-1}\theta_i|_{\Gal_K}$. Since the curve $\E$ has
  additive reduction only at primes dividing $2$, both $\psi_1$ and
  $\psi_2$ are unramified outside primes dividing $2$ and
  $p$. Furthermore, by \cite[Lemma 1]{MR2371778}, one of the
  characters is unramified outside $p$ (say $\psi_1$).

  The prime $2$ splits in $\Q(\sqrt{129})/\Q$, say
  $(2) = \id{p}_2\bar{\id{p}}_2$. By \cite[Lemma 2.8]{PT}, the conductor
  of $\E$ at $(\id{p}_2,\bar{\id{p}}_2)$ equals one of $(8,8), (1,6)$
  or $(4,6)$, hence the character $\psi_1$ has conductor at most
  $2^4, \id{p}_2^3$ or $4\cdot \id{p}_2$ (or their conjugates). The
  ray class group for such conductors has exponent $4$ in the first
  case and $2$ in the other two cases (computed using \cite{PARI2}). In
  particular the curve (or a quadratic twist of it) has a rational
  point over an extension of degree $2$ or $4$ over $\Q$, hence
  $p \le 17$ by \cite[Theorem 1.2]{1707.00364}.
		
  Theorem~\ref{thm:levelandnebentypus}, Ribet's lowering the level result and the proof of \cite[Lemma 2.8]{PT}
  imply that $\rho_{\E,p}\otimes\chi$ is congruent modulo $p$ to the Galois representation of a newform in
  $S_2(2 \cdot 3\cdot 43,\varepsilon)$ (when $c$ is even) or in 
  $S_2(2^8\cdot 3\cdot 43,\varepsilon)$ (when $c$ is odd), where
  $\varepsilon$ corresponds to $\Q(\sqrt{129})$.

\vspace{1pt}
\noindent $\bullet$ The space $S_2(2 \cdot 3\cdot 43,\varepsilon)$ has $4$ Galois conjugacy classes, none of them having complex multiplication. Using Mazur's trick for primes $5\le q\le 20$, all newforms in the first space can be discarded assuming $p>5$.

\vspace{1pt}
\noindent $\bullet$ The space $S_2(2^8\cdot 3\cdot 43,\varepsilon)$
has $36$ Galois conjugacy classes, $18$ of them having complex
multiplication. Using Mazur's trick for primes $5\le q\le 20$, the
first $33$ newforms (in \verb*|Magma|'s order) can be discarded assuming $p
\notin\{2,5,7,11,13,17,23,43\}$, but four newforms having complex multiplication by $\ZZ[\sqrt{-2}]$. The last three newforms do not have complex multiplication, but they do have a large
coefficient field and \verb*|Magma| is unable to compute norms over
these fields, so we used \verb*|Magma| to compute the coefficients
$a_5$ and $a_7$ of each of these newforms and apply Mazur's trick in
\verb*|PARI/GP| for $q = 5, 7$ by hand (where the norms are computed
within a few seconds). It follows that they can be discarded if
$p\not\in\{2,5,7,37\}$.

Since in this case $2$ splits over $K$, then we can use the results of
Section~\ref{section:Ellenberg} to discard the newforms having complex
multiplication.  After a computer search for the minimum $x$ we
obtained that taking $x = 49885$ in ~(\ref{eq:final}) (using the
inequalities (\ref{eq:boundA}) and (\ref{eq:5})) makes the right hand
side positive for $p > 900$. This can be checked with the following
command (in \verb*|PARI/GP|):
\begin{verbatim}
? read("RemoveCM");
? Bound(129,907,49885)
%2 = 0.039412707010082109791157365950637933812
\end{verbatim}
For small primes, the same argument as in the previous examples works;
note that $c$ is divisible by an odd prime larger than $3$ because it
cannot be divisible by $3$ (as the solution is primitive) and it is
not divisible by $2$ because the modular forms with complex
multiplication appear in the space
$S_2(2^8\cdot 3\cdot 43,\varepsilon)$.
 Then we are again in the hypothesis of \cite[Proposition
3.4]{MR2075481}, which discards newforms with complex
multiplication by $\Z[\sqrt{-2}]$ for primes $p\equiv 1,3\pmod8$.
\end{proof}
\begin{remark} Ellenberg's bound obtained in the last example could
  probably be slightly improved if better bounds are given in the
  computations of Section~\ref{section:Ellenberg}. If the final value
  is not too large, a newform $f \in S_2(2p^2)$ with the desired
  properties could be found in the intermediate range via a computer
  search.
	\end{remark}
	
	\bibliographystyle{alpha} \bibliography{biblio}
	
\end{document}